\documentclass[12pt, reqno]{article}
\usepackage{tikz}
\usepackage{amssymb, mathrsfs}
\usepackage{latexsym}
\usepackage{amsmath}
\usepackage{amsthm}
\usepackage{amsfonts}
\usepackage{dsfont}
\usepackage{mathtools}
\usepackage{epsfig}
\usepackage{verbatim}
\usepackage[all,cmtip]{xy}
\usepackage[left=1.2in,right=1.2in,top=1in,bottom=1in]{geometry}
\usepackage{enumerate}
\usepackage[colorlinks=true,linkcolor=blue,citecolor=blue]{hyperref}

\usepackage{authblk}

\newtheorem*{lemma*}{Lemma}

\newtheorem{theorem}{Theorem}[section]
\newtheorem{proposition}[theorem]{Proposition}
\newtheorem{lemma}[theorem]{Lemma}
\newtheorem*{thma}{Theorem A}
\newtheorem*{thmb}{Theorem B}
\newtheorem{corollary}[theorem]{Corollary}

\theoremstyle{definition}
\newtheorem{claim}[theorem]{Claim}
\newtheorem*{claim*}{Claim}
\newtheorem{definition}[theorem]{Definition}
\newtheorem{example}[theorem]{Example}

\theoremstyle{remark}
\newtheorem{remark}[theorem]{Remark}

\newcommand{\ind}{\textup{Ind}}
\newcommand{\supp}{\textup{Supp}}
\newcommand{\ev}{\textup{ev}}
\newcommand{\diam}{\textup{diameter}}
\newcommand{\cl}{ \textup{C}\ell} % real Clifford
\newcommand{\ccl}{\hspace{1pt}\mathbb C\ell}  % complex Clifford
\newcommand{\spin}{\textup{spin}}
\newcommand{\pos}{\textup{Pos}^\spin}
\newcommand{\sn}{\varepsilon}

\begin{document}
\title{Positive scalar curvature, higher rho invariants and localization algebras}

\author{Zhizhang Xie\thanks{Email: \texttt{xie@math.tamu.edu}} }
\author{Guoliang Yu\thanks{Email: \texttt{guoliangyu@math.tamu.edu}; partially supported by the US National Science Foundation.}}
\affil{Department of Mathematics, Texas A\&M University}
\date{}

\maketitle

\begin{abstract}
In this paper, we use localization algebras to study higher rho invariants of closed spin manifolds with positive scalar curvature metrics. The higher rho invariant is a secondary invariant and is closely related to positive scalar curvature problems. The main result of the paper connects the higher index of the Dirac operator on a spin manifold with boundary to the higher rho invariant of the Dirac operator on the boundary,  where the boundary is endowed with a positive scalar curvature metric. Our result extends a theorem of Piazza and Schick \cite[Theorem 1.17]{Piazza:2012fk}. 
\end{abstract}

\section{Introduction}

Ever since the index theorem of Atiyah and Singer \cite{MAIS63}, the main goal of index theory has been to associate (topological or geometric) invariants to operators on manifolds. The Atiyah-Singer index theorem states that the analytic index of an elliptic operator on a closed manifold (i.e. compact manifold without boundary) is equal to the topological index of this operator, where the latter can be expressed in terms of topological information of the underlying manifold. The corresponding index theorem for manifolds with boundary was proved by Atiyah, Patodi and Singer \cite{APS73b}. Due to the presence of the boundary, a secondary invariant naturally appears in the index formula. This secondary invariant depends only on the boundary and is called the eta invariant \cite{A-P-S75a}. 

When one deals with (closed) manifolds that have nontrivial fundamental groups, it turns out that one has a more refined index theory, where the index takes values in the $K$-theory group of certain $C^\ast$-algebra of the fundamental group. This is now referred to as the higher index theory. One of its main motivations is the Novikov conjecture, which states that higher signatures are homotopy invariants of manifolds. Using techniques from higher index theory to study the Novikov conjecture traces back to the work of Mi{\v{s}}{\v{c}}enko \cite{MR0362407}. With the seminal work of Kasparov \cite{GK88}, and Connes and Moscovici \cite{CM90}, this line of development has turned out to be fruitful and fundamental in the study of the Novikov conjecture. Further development was pursued by many authors, as part of the noncommutative geometry program initiated by Connes \cite{AC85, AC94}.  

As the higher index theory for closed manifolds has important applications to topology and geometry, one naturally asks whether the same works for manifolds with boundary. In other words, we would like to have a higher version of the Atiyah-Patodi-Singer index theorem for manifolds with boundary. This turns out to be a difficult task and we are  still far away from having a complete picture. In particular, a higher version of the eta invariant arises naturally as a secondary invariant. However, we do not know how to define it in general.  This higher eta invariant was first studied by Lott \cite{JLott92}. For a fixed closed spin manifold $M$, the higher eta invariant is an obstruction for two positive scalar curvature metrics on $M$ to lie in the the same connected component of $\mathcal R^+(M)$, the space of all positive scalar curvature metrics on $M$. 

The connection of higher index theory to positive scalar curvature problems was already made apparent in the work of Rosenberg (cf. \cite{JR83}). The applications of the higher eta invariant to the study of  positive scalar curvature metrics on manifolds can be found in the work of Leitchnam and Piazza (cf. \cite{ELPP01}).   The higher rho invariant is a variant of the higher eta invariant. It was also first investigated by Lott in the cyclic cohomological setting \cite{JLott92} (see also Weinberger's paper for a more topological approach \cite{MR1707352}). In a recent paper of Piazza and Schick \cite{Piazza:2012fk}, they studied the higher rho invariant on spin manifolds equipped with positive scalar curvature metrics, in the $K$-theoretic setting. The higher rho invariant was first introduced by Higson and Roe \cite[Definition 7.1]{MR2761858}. In particular,   their higher rho invariant lies in the $K$-theory group of a certain $C^\ast$-algebra.  

In this paper, we use localization algebras (introduced by the second author \cite{MR1451759,GY98})  to study the higher rho invariant on spin manifolds equipped with positive scalar curvature metrics. Let $N$ be a spin manifold with boundary, where the boundary $\partial N$ is endowed with a positive scalar curvature metric.  In the main theorem of this paper, we show that the $K$-theoretic ``boundary'' of the higher index class of the Dirac operator on $N$ is identical to the higher rho invariant of the Dirac operator on  $\partial N$.  More generally, let $M$ be an $m$-dimensional complete spin manifold with boundary $\partial M$ such that
\begin{enumerate}[(i)]
\item the metric on $M$ has product structure near $\partial M$ and its restriction on $\partial M$, denoted by $h$, has positive scalar curvature;
\item there is a proper and cocompact isometric action of a discrete group $\Gamma$ on $M$;
\item the action of $\Gamma$ preserves the spin structure of $M$. 
\end{enumerate}  
We denote the associated Dirac operator on $M$ by $D_{M}$ and the associated Dirac operator on $\partial M$ by $D_{\partial M}$. With the positive scalar curvature metric $h$ on the boundary $\partial M$, we can naturally define the higher index class $\ind(D_M)$ of $D_M$ and the higher rho invariant $\rho(D_{\partial M}, h)$ of $D_{\partial M}$ (see Section $\ref{sec:roelocal}$ and $\ref{sec:indexmap}$ for the precise definitions). This higher rho invariant $\rho(D_{\partial M}, h)$ lives in the $K$-theory group $K_{m-1}(C^\ast_{L,0}(\partial M)^\Gamma) $. Here $C^\ast_{L,0}(\partial M)^\Gamma$ is the kernel of the evaluation map $\ev: C_L^\ast(\partial M)^\Gamma \to C^\ast(\partial M)^\Gamma$ (see Section $\ref{sec:roelocal}$ and $\ref{sec:indexmap}$ below for the precise definitions). In fact, for a proper metric space $X$ equipped with a proper and cocompact isometric action of a discrete group $\Gamma$, we have $K_{n+1}(C^\ast_L(X)^\Gamma) \cong K_{n+1}^\Gamma(X)$, where the latter group is the equivariant $K$-homology with $\Gamma$-compact supports of $X$ \cite[Theorem 3.2]{MR1451759}. Moreover, $K_{n+1}(C_{L,0}^\ast(X)^\Gamma) \cong K_{n}(D^\ast(X)^\Gamma)$, where $D^\ast(X)^\Gamma$ is Roe's structure algebra of $X$ (see Section $\ref{sec:roelocal}$ for the precise definition).  Notice that the short exact sequence of $C^\ast$-algebras
\[  0\to C^\ast_{L,0}(M)^\Gamma \to C^\ast_L(M)^\Gamma \to C^\ast(M)^\Gamma \to 0\]
induces the following long exact sequence in $K$-theory:
\[ \cdots \to K_i(C_L^\ast(M)^\Gamma) \to  K_i(C^\ast(M)^\Gamma) \xrightarrow{\partial_i } K_{i-1}(C^\ast_{L,0}(M)^\Gamma) \to  K_{i-1}(C^\ast_L(M)^\Gamma) \to \cdots. \]
Also, by functoriality, we have a natural homomorphism 
\[ \iota_\ast: K_{m-1}(C^\ast_{L,0}(\partial M)^\Gamma)  \to K_{m-1}(C^\ast_{L,0}(M)^\Gamma) \]
induced by the inclusion map $\iota: \partial M \hookrightarrow M$. With the above notation, we have  the following main theorem of the paper. 
\begin{thma} 
\[ \partial_m(\ind(D_{M}) )  = \iota_\ast( \rho(D_{\partial M}, h)) \]
in $K_{m-1}(C^\ast_{L,0}(M)^\Gamma)$. 
\end{thma} 
 This theorem extends a theorem of Piazza and Schick \cite[Theorem 1.17]{Piazza:2012fk} to all dimensions, and to both complex and real cases. As an immediate application, one sees that nonvanishing of the higher rho invariant is an obstruction to extension of the positive scalar curvature metric from the boundary to the whole manifold (cf. Corollary $\ref{cor:ext}$ below). Moreover, the higher rho invariant can be used to distinguish whether or not two positive scalar curvature metrics are connected by a path of positive scalar curvature metrics (cf. Corollary $\ref{cor:diff}$ below). In a similar context, these have already appeared in the work of Lott \cite{JLott92}, Botvinnik and Gilkey \cite{MR1339924}, and Leichtnam and Piazza \cite{ELPP01}.

We shall also use the theorem above to map Stolz' positive scalar curvature exact sequence to the exact sequence of Higson and Roe.  Recall that Stolz introduced in \cite{SS95}  the following positive scalar curvature exact sequence 
\[ 
\xymatrix{  \ar[r] &\Omega^\spin_{n+1}(B\Gamma) \ar[r] & R_{n+1}^{\spin}(B\Gamma) \ar[r]  & \pos_n(B\Gamma) \ar[r]   &\Omega_n^\spin (B\Gamma) \ar[r]  &  }
\] 
where $B\Gamma$ is the classifying space of a discrete group $\Gamma$. Moreover, $\Omega^\spin_n(B\Gamma)$ is the spin bordism group,   $\pos_n(B\Gamma)$ is certain structure group of positive scalar curvature metrics and $R_{n+1}^\spin(B\Gamma)$ is certain obstruction group for the existence of positive scalar curvature metric (see Section $\ref{sec:stolz}$ below for the precise definitions). This is analogous to the surgery exact sequence in topology. In fact, for the surgery exact sequence, Higson and Roe constructed a natural homomorphism  \cite{MR2220522,MR2220523,MR2220524} from the surgery exact sequence to the following exact sequence of $K$-theory of $C^\ast$-algebras:
\[ \to K_{n+1}(B\Gamma) \to K_{n+1}(C_r^\ast(\Gamma)) \to K_{n+1}(D^\ast(\Gamma) ) \to K_n(B\Gamma) \to  \]  
where $C_r^\ast(\Gamma)$ is the reduced $C^\ast$-algebra of $\Gamma$ and \[ K_{n+1}(D^\ast(\Gamma) ) = \varinjlim_{X\subseteq E\Gamma} K_{n+1}( D^\ast(X)^\Gamma)\] where $E\Gamma$ is the universal cover of $B\Gamma$ and $X$ runs through all closed $\Gamma$-invariant subcomplexes of $E\Gamma$ such that $X/\Gamma$ is compact. Hence this exact sequence of Higson and Roe provides natural index theoretic invariants for the surgery exact sequence. Moreover, it is closely related to the Baum-Connes conjecture.

It is a natural task to construct a similar homomorphism from the Stolz' positive scalar curvature exact sequence to the Higson-Roe exact sequence. This was first taken up in a recent paper of Piazza and Schick \cite{Piazza:2012fk}, where they showed that the following diagram commutes when $n+1$ is even:
\[ 
\xymatrix{  \ar[r] &\Omega^\spin_{n+1}(B\Gamma) \ar[r] \ar[d] & R_{n+1}^{\spin}(B\Gamma) \ar[r] \ar[d] & \pos_n(B\Gamma) \ar[r] \ar[d]  &\Omega_n^\spin (B\Gamma) \ar[r]\ar[d]  & \\
\ar[r] &  K_{n+1}(B\Gamma) \ar[r]  &  K_{n+1}(C_r^\ast(\Gamma)) \ar[r]  &  K_{n+1}^\Gamma(D^\ast(\Gamma) ) \ar[r]  & K_n(B\Gamma) \ar[r]  & }
\] 
Here all the vertical maps are naturally defined (cf. Section $\ref{sec:stolz}$ and $\ref{sec:coin}$). 

Of course, one would expect the above diagram to commute for all dimensions, regardless of the parity of $n+1$. In this paper, we show that this is indeed the case. Regarding the method of proof, our approach appears to be more conceptual than that of Piazza and Schick. One of the key ingredients is the use of Kasparov $KK$-theory \cite{GK80} (cf. Section $\ref{sec:kk}$). In particular, our proofs works equally well for both the even and the odd cases.  More precisely, we have the following result.  
\begin{thmb}
Let $X$ be a proper metric space equipped with a proper and cocompact isometric action of a discrete group $\Gamma$. For all $n\in \mathbb N$,  the following diagram commutes
\[ \xymatrix{ \Omega^{\spin}_{n+1}(X)^\Gamma \ar[d]^{\ind_L} \ar[r] & R_{n+1}^{\spin}(X)^\Gamma \ar[r] \ar[d]^{\ind} & \pos_n(X)^\Gamma \ar[r] \ar[d]^{\rho}  &\Omega_{n}^{\spin}(X)^\Gamma  \ar[d]^{\ind_L}  \\
 K_{n+1}(C^\ast_L(X)^\Gamma) \ar[r] & K_{n+1} (C^\ast(X)^\Gamma) \ar[r]^{\partial} & K_{n}(C^\ast_{L, 0}(X)^\Gamma) \ar[r] & K_{n}(C^\ast_L(X)^\Gamma)   }
\]
\end{thmb}
One of the main concepts that we use is the notion of localization algebras. We refer the reader to  Section $\ref{sec:pre}$ and Section $\ref{sec:stolz}$ for the precise definitions of various terms.  We point out that the second row of the above diagram is canonically equivalent to the long exact sequence of Higson and Roe (see Proposition $\ref{prop:coin}$ below for a proof). 

Now if $X = \underline{E}\Gamma,$ the universal space of $\Gamma$-proper actions, then  the map 
\[ K_{n+1}(C^\ast_L(X)^\Gamma) \to K_{n+1} (C^\ast(X)^\Gamma)\]
 in the above commutative diagram can be naturally identified with the Baum-Connes assembly map (cf. \cite{BCH94}). In principle, the higher rho invariant could provide nontrivial secondary invariants that do not lie in the image of the Baum-Connes assembly map.  

The paper is organized as follows. In Section $\ref{sec:pre}$, we recall the definitions of various basic concepts that will be used later in the paper. In Section $\ref{sec:mbpsc}$, we carry out a detailed construction of the higher index class of the Dirac operator on a manifold whose boundary is equipped with a positive scalar curvature metric. In Section $\ref{sec:main}$, we prove the main theorem of the paper. In Section $\ref{sec:stolz}$, we apply our main theorem to map the Stolz' positive scalar curvature exact sequence to a long exact sequence involving the $K$-theory of localization algebras. In Section $\ref{sec:coin}$, we show that the $K$-theoretic long exact sequence used in Section $\ref{sec:stolz}$ is canonically isomorphic to the long exact sequence of Higson and Roe. In particular, the explicit construction shows that the definition of the higher rho invariant in our paper is naturally identical to that of Higson and Roe.

\vspace{.5cm}
\noindent\textbf{Acknowledgements.} We wish to thank Thomas Schick for bringing the main question
of the paper to our attention and for sharing with us a draft of his joint paper with Paolo Piazza \cite{Piazza:2012fk}. We would also like to thank both Herv\'{e} Oyono-Oyono and Thomas Schick for useful comments and suggestions.

\section{Preliminaries}\label{sec:pre}

\subsection{K-theory and index maps}\label{sec:kt}
For a short exact sequence of $C^\ast$-algebras $ 0 \to  \mathcal J \to \mathcal A \to \mathcal A/\mathcal J \to 0$,
we have a six-term exact sequence in $K$-theory:
\[ 
\xymatrix { K_0( \mathcal J ) \ar[r] & K_0(\mathcal A) \ar[r] & K_0(\mathcal A/\mathcal J )  \ar[d]^{\partial_1} \\
K_1(\mathcal A/\mathcal J) \ar[u]^{\partial_0} 
  & K_1(\mathcal A) \ar[l] & K_1(\mathcal J) \ar[l]
}
\]  
Let us recall the definition of the boundary maps $\partial_i$. 

\textbf{Even case.} Let $u$ be an invertible element in $\mathcal A/\mathcal J$. Let $v$ be the inverse of $u$ in $\mathcal A/\mathcal J$. Now suppose $U, V\in \mathcal A$ are lifts of $u$ and $v$. We define 
\[ W = \begin{pmatrix} 1 & U\\ 0 & 1\end{pmatrix} \begin{pmatrix} 1 & 0 \\ -V & 1\end{pmatrix} \begin{pmatrix} 1 & U  \\ 0 & 1 \end{pmatrix}\begin{pmatrix} 0 & -1\\ 1 & 0 \end{pmatrix}. \]
Notice that $W$ is invertible and a direct computation shows that 
\[  W - \begin{pmatrix} U & 0 \\ 0 & V \end{pmatrix} \in  M_2(\mathbb C)\otimes \mathcal J.\]
Consider the idempotent
\begin{equation}\label{eq:keven}
 P = W \begin{pmatrix} 1 & 0 \\ 0 & 0\end{pmatrix} W^{-1} = \begin{pmatrix} UV + UV(1-UV) & (2 + UV)(1-UV) U \\ V(1-UV) & (1-UV)^2\end{pmatrix}.
\end{equation}
We have  
\[ P - \begin{pmatrix} 1 & 0 \\0 & 0\end{pmatrix} \in M_2(\mathbb C)\otimes \mathcal J. \]
By definition, 
\[  \partial_0 ([u]) = [P] - \left[\begin{pmatrix} 1 & 0 \\0 & 0\end{pmatrix} \right] \in K_0(\mathcal J).\]

\begin{remark}
If $u$ is unitary in $\mathcal A/\mathcal J$ and $U\in \mathcal A$ is a lift of $u$, then we can choose $V = U^\ast$. 
\end{remark}

\textbf{Odd case.} Let $q $ be an idempotent in $\mathcal A/\mathcal J$ and $Q$ a lift of $q$ in $\mathcal A$. Then 
\[  \partial_1([q])  = [e^{2\pi iQ}] \in K_1(\mathcal J). \] 
 
\subsection{Roe algebras and localization algebras}\label{sec:roelocal}
In this subsection, we briefly recall some standard definitions in coarse geometry. We refer the reader to \cite{MR1147350, MR1451759} for more details. Let $X$ be a proper metric space. That is, every closed ball in $X$ is compact. An	$X$-module is a separable Hilbert space equipped with a	$\ast$-representation of $C_0(X)$, the algebra of all continuous functions on $X$ which vanish at infinity. An	$X$-module is called nondegenerate if the $\ast$-representation of $C_0(X)$ is nondegenerate. An $X$-module is said to be standard if no nonzero function in $C_0(X)$ acts as a compact operator. 
\begin{definition}
Let $H_X$ be a $X$-module and $T$ a bounded linear operator acting on $H_X$. 
\begin{enumerate}[(i)]
\item The propagation of $T$ is defined to be $\sup\{ d(x, y)\mid (x, y)\in \supp(T)\}$, where $\supp(T)$ is  the complement (in $X\times X$) of the set of points $(x, y)\in X\times X$ for which there exist $f, g\in C_0(X)$ such that $gTf= 0$ and $f(x)\neq 0$, $g(y) \neq 0$;
\item $T$ is said to be locally compact if $fT$ and $Tf$ are compact for all $f\in C_0(X)$; 
\item $T$ is said to be pseudo-local if $[T, f]$ is compact for all $f\in C_0(X)$.  
\end{enumerate}
\end{definition}

\begin{definition}
Let $H_X$ be a standard nondegenerate $X$-module and $\mathcal B(H_X)$ the set of all bounded linear operators on $H_X$.  
\begin{enumerate}[(i)]
\item The Roe algebra of $X$, denoted by $C^\ast(X)$, is the $C^\ast$-algebra generated by all locally compact operators with finite propagations in $\mathcal B(H_X)$.
\item $D^\ast(X)$ is the $C^\ast$-algebra generated by all pseudo-local operators with finite propagations in $\mathcal B(H_X)$. In particular, $D^\ast(X)$ is a subalgebra of the multiplier algebra of $C^\ast(X)$.
\item $C_L^\ast(X)$ (resp. $D_{L}^\ast(X)$) is the $C^\ast$-algebra generated by all bounded and uniformly norm-continuous functions $f: [0, \infty) \to C^\ast(X)$ (resp. $f: [0, \infty) \to D^\ast(X)$)  such that 
\[ \textup{propagation of $f(t) \to 0 $, as $t\to \infty$. }\]  
Again $D_{L}^\ast(X)$ is a subalgebra of the multiplier algebra of $C_L^\ast(X)$. 
\item $C_{L, 0}^\ast(X)$ is the kernel of the evaluation map 
\[  \ev : C_L^\ast(X) \to C^\ast(X),  \quad   \ev (f) = f(0).\]
In particular, $C_{L, 0}^\ast(X)$ is an  ideal of $C_L^\ast(X)$. Similarly, we define $D_{L, 0}^\ast(X)$ as the kernel of the evaluation map from $D_L^\ast(X)$ to $D^\ast(X)$.
\item If $Y$ is a subspace of $X$, then $C_L^\ast(Y; X)$ (resp. $C_{L,0}^\ast(Y;X)$) is defined to be the closed subalgebra of $C_L^\ast(X)$ (resp. $C_{L,0}^\ast(X)$) generated by all elements $f$ such that there exist $c_t>0$ satisfying $ \lim_{t\to \infty} c_t = 0$ and $\supp(f(t)) \subset \{ (x, y) \in X\times X \mid d((x,y), Y\times Y) \leq c_t\}$ for all $t$.  
 
\end{enumerate}
\end{definition}

Now in addition we assume that a discrete group $\Gamma$ acts properly and cocompactly on  $X$  by isometries. In particular, if the action of $\Gamma$ is free, then $X$ is simply a $\Gamma$-covering of the compact space $X/\Gamma$. 

Now let $H_X$ be a $X$-module equipped with a covariant unitary representation of $\Gamma$. If we denote the representation of $C_0(X)$ by $\varphi$ and the representation of $\Gamma$ by $\pi$, this means 
\[  \pi(\gamma) (\varphi(f) v )  =  \varphi(f^\gamma) (\pi(\gamma) v),\] 
where $f\in C_0(X)$, $\gamma\in \Gamma$, $v\in H_X$ and $f^\gamma (x) = f (\gamma^{-1} x)$. In this case, we call $(H_X, \Gamma, \varphi)$ a covariant system.  

\begin{definition}[\cite{MR2732068}]
A covariant system $(H_X, \Gamma, \varphi)$ is called admissible if 
\begin{enumerate}[(1)]
\item the $\Gamma$-action on $X$ is proper and cocompact;
\item $H_X$ is a nondegenerate standard $X$-module;
\item for each $x\in X$, the stabilizer group $\Gamma_x$ acts on $H_X$ regularly in the sense that the action is isomorphic to the action of $\Gamma_x$ on $l^2(\Gamma_x)\otimes H$ for some infinite dimensional Hilbert space $H$. Here $\Gamma_x$ acts on $l^2(\Gamma_x)$ by translations and acts on $H$ trivially. 
\end{enumerate}
\end{definition}
We remark that for each locally compact metric space $X$ with a proper and cocompact isometric action of $\Gamma$, there exists an admissible covariant system $(H_X, \Gamma, \varphi)$. Also, we point out that the condition $(3)$ above is automatically satisfied if $\Gamma$ acts freely on $X$. If no confusion arises, we will denote an admissible covariant system $(H_X, \Gamma, \varphi)$ by $H_X$ and call it an admissible $(X, \Gamma)$-module. 

\begin{remark}\label{rm:ad}
For each $(X, \Gamma)$ above, there always exists an admissible $(X, \Gamma)$-module $\mathcal H$. In particular, $H\oplus \mathcal H$ is an admissible $(X, \Gamma)$-module for every $(X, \Gamma)$-module $H$. 
\end{remark}

\begin{definition}
Let $X$ be a locally compact metric space $X$ with a proper and cocompact isometric action of $\Gamma$. If $H_X$ is an admissible $(X, \Gamma)$-module, we denote by $\mathbb C[X]^\Gamma$ the $\ast$-algebra of all $\Gamma$-invariant locally compact operators with finite propagations in $\mathcal B(H_{X})$.  We define $C^\ast(X)^\Gamma$ to be the completion of $\mathbb C[X]^\Gamma$ in $\mathcal B(H_{ X})$.
\end{definition}
Since the action of $\Gamma$ on $X$ is cocompact, it is known that in this case $C^\ast(X)^\Gamma$ is $\ast$-isomorphic to $C^\ast_r(\Gamma)\otimes \mathcal K$, where $ C^\ast_r(\Gamma)$ is the reduced group $C^\ast$-algebra of $\Gamma$ and $\mathcal K$ is the algebra of all compact operators.

Similarly, we can also define  $D^\ast( X)^\Gamma$, $C_L^\ast( X)^\Gamma$, $D_L^\ast( X)^\Gamma$,  $C_{L,0}^\ast( X)^\Gamma$, $D_{L, 0}^\ast( X)^\Gamma$, $C^\ast_{L}( Y;  X)^\Gamma$ and $C^\ast_{L,0}( Y;  X)^\Gamma$.   

\begin{remark}
Up to isomorphism, $C^\ast(X) = C^\ast(X, H_X)$ does not depend on the choice of the standard nondegenerate $X$-module $H_X$. The same holds for $D^\ast(X)$, $C_L^\ast(X)$, $D_L^\ast( X)$,  $C_{L,0}^\ast( X)$, $D_{L, 0}^\ast(X)$, $C^\ast_{L}( Y;X)$, $C^\ast_{L,0}( Y;X)$ and their $\Gamma$-invariant versions. 
\end{remark}

\begin{remark}
Note that we can also define maximal versions of all the $C^\ast$-algebras above. For example, we define the maximal $\Gamma$-invariant Roe algebra $C^\ast_{\max}(X)^\Gamma$ to be the completion of $\mathbb C[X]^\Gamma$ under the maximal norm:
\[  \|a\|_{\max} = \sup_{\phi} \ \big\{\|\phi(a)\| \mid \phi: \mathbb C[X]^\Gamma \to \mathcal B(H') \ \textup{a $\ast$-representation} \big\}. \]
\end{remark}

\subsection{Index map, local index map, higher rho invariant and Kasparov $KK$-theory}\label{sec:indexmap}
In this subsection, we recall the constructions of the index map, the  local index map (cf. \cite{MR1451759, MR2732068})  and the higher rho invariant.

Let $X$ be a locally compact metric space with a proper and cocompact isometric action of $\Gamma$. We recall the definition of the $K$-homology groups $K_j^\Gamma(X)$, $j=0, 1$. They are generated by certain cycles modulo certain equivalence relations (cf. \cite{GK88}):
\begin{enumerate}[(i)]
\item an even cycle for $K_0^\Gamma(X)$ is a pair $(H_X, F)$, where $H_X$ is an admissible $(X, \Gamma)$-module and $F\in \mathcal B(H_X)$ such that $F$ is $\Gamma$-equivariant,     $F^\ast F - I$ and $F F^\ast - I$ are locally compact and $[F, f] = F f - fF $ is compact for all $f\in C_0(X)$.   
\item an odd cycle for $K_1^\Gamma(X)$ is a pair $(H_X, F)$, where $H_X$ is an admissible $(X, \Gamma)$-module and $F$ is a $\Gamma$-equivariant self-adjoint operator in $\mathcal B(H_X)$ such that $F^2 - I$ is locally compact and $[F, f]$ is compact for all $f\in C_0(X)$. 
\end{enumerate}

\begin{remark}\label{rm:ad2}
In fact, we get the same group even if  we allow cycles $(H_X, F)$ with possibly non-admissible $H_X$ in the above definition. Indeed,  we can take the direct sum $H_X$ with an admissible $(X, \Gamma)$-module $\mathcal H$ and define $F' = F \oplus 1$. It is easy to see that $(H_X\oplus \mathcal H, F')$ is equivalent to $(H_X, F)$ in $K_0^\Gamma(X)$.  So from now on, we assume $H_X$ is an admissible $(X, \Gamma)$-module.  
\end{remark}

\begin{remark}
In the general case where the action of $\Gamma$ on $X$ is not necessarily cocompact, we define 
\[  K_i^\Gamma(X)  = \varinjlim_{Y\subseteq X}  K_i^\Gamma(Y) \]
where $Y$ runs through all closed $\Gamma$-invariant subsets of $X$ such that $Y/\Gamma$ is compact. 
\end{remark}
  
\subsubsection{Index map and local index map}\label{sec:localind}
Now let  $(H_X, F)$ be an even cycle for $K_0^\Gamma(X)$. Let $\{U_i\}$ be a $\Gamma$-invariant locally finite open cover of $X$ with diameter$(U_i) < c$ for some fixed $c > 0$. Let $\{\phi_i\} $ be a $\Gamma$-invariant continuous partition of unity subordinate to $\{U_i\}$. We define
\[  G = \sum_{i} \phi^{1/2}_i F \phi^{1/2}_i,\]
where the sum converges in strong topology. It is not difficult to see that $(H_X, G)$ is equivalent to $(H_X, F)$ in $K^\Gamma_0(X)$. By using the fact that $G$ has finite propagation, we see that $G$ is a multiplier of $C^\ast(X)^\Gamma$ and $G$ is a unitary modulo $C^\ast(X)^\Gamma$. Now by the standard construction in Section $\ref{sec:kt}$ above, $G$ produces a class $[G] \in K_0(C^\ast(X)^\Gamma)$. We define the index of $(H_X, F)$ to be $[G]$.

From now on, we denote this index class of $(H_X, F)$ by $\ind(H_X, F)$ or simply $\ind(F)$ if no confusion arises.

To define the local index of $(H_X, F)$, we need to use a family of partitions of unity. More precisely, for each $n\in \mathbb N$, let  $\{U_{n, j}\}$ be a $\Gamma$-invariant locally finite open cover of $X$ with diameter $(U_{n,j}) < 1/n$ and $\{\phi_{n, j}\}$ be a $\Gamma$-invariant continuous partition of unity subordinate to $\{U_{n, j}\}$. We define 
\begin{equation}\label{eq:path}
 G(t) = \sum_{j} (1 - (t-n)) \phi_{n, j}^{1/2} G \phi_{n,j}^{1/2} + (t-n) \phi_{n+1, j}^{1/2} G \phi_{n+1, j}^{1/2}
\end{equation}
for $t\in [n, n+1]$. 
\begin{remark}
Here by convention, we assume that the open cover $\{U_{0, j}\}$ is the trivial cover $\{ X\}$ when $n = 0$. 
\end{remark}
Then $G(t), 0\leq t <\infty,$ is a multiplier of $C^\ast_L(X)^\Gamma$ and a unitary modulo $C^\ast_L(X)^\Gamma$.  Hence by the  construction in Section $\ref{sec:kt}$ above, $G(t)$ produces a $K$-theory class  $ [G(t)] \in K_0(C^\ast_L(X)^\Gamma)$.  We call this $K$-theory class  the local index class of $(H_X, F)$. If no confusion arises,  we denote this local index class of $(H_X, F)$ by $\ind_L(H_X, F)$ or simply $\ind_L(F)$ from now on.

Now let $(H_X, F)$ be an odd cycle in $K_1^\Gamma(X)$. With the same notation as above, we set $q = \frac{G + 1}{2}$. Then the index class of $(H_X, F)$ is defined to be $ [e^{2\pi i q}]\in K_1(C^\ast(X)^\Gamma).$ For the local index of $(H_X, F)$, one simply uses $q(t) = \frac{G(t) + 1}{2}$ in place of $q$. 

\begin{remark}
For a locally compact metric space $X$ with a proper and cocompact isometric action of $\Gamma$, the local index map  induces a canonical isomorphism $\ind_L: K_i^\Gamma(X)\xrightarrow{\cong}K_i(C_L^\ast(X)^\Gamma)$
 \cite[Theorem 3.2]{MR1451759}. In fact, by the work of Qiao-Roe \cite{MR2661442}, the isomorphism holds true for more general spaces without the cocompact assumption. 
\end{remark}

Now suppose $M$ is an odd dimensional complete spin manifold without boundary. Assume that there is a discrete group $\Gamma$ acting on $M$  properly and cocompactly by isometries. In addition, we assume the action of $\Gamma$ preserves the spin structure on $M$. Let $S$ be the spinor bundle over $M$ and $D= D_{M}$ be the associated Dirac operator on $M$.  Let $H_M = L^2(M, S) $ and  
\[  F = D( D^2 + 1)^{-1/2}.  \]
Then $(H_M, F)$ defines a class in $K_1^\Gamma(M)$. Note that in fact $F$ lies in the multiplier algebra of $C^\ast(M)^\Gamma$, since $F$ can be approximated\footnote{This can be achieved by choosing an appropriate sequence of smooth partitions of unity.} by elements of finite propagation in the multiplier algebra of $C^\ast(M)^\Gamma$. As a result, we can directly work with\footnote{In other words, there is no need to pass to the operator $G$ or $G(t)$ as in the general case.} 
\begin{equation}\label{eq:path2}
 F(t) = \sum_{j} (1 - (t-n)) \phi_{n, j}^{1/2} F \phi_{n,j}^{1/2} + (t-n) \phi_{n+1, j}^{1/2} F \phi_{n+1, j}^{1/2}
\end{equation}
for $t\in [n, n+1]$. And the same argument above defines the index class and the local index class of $(H_M, F)$. We shall denote them by $\ind(D_M)\in K_1(C^\ast(M)^\Gamma)$ and $\ind_L(D_M)\in K_1(C_L^\ast(M)^\Gamma)$ respectively.

The even dimensional case is essentially the same, where one needs to work with the natural $\mathbb Z/2\mathbb Z$-grading on the spinor bundle.  We leave the details to the interested reader (see also Section $\ref{sec:idem}$ below).

\begin{remark}
If we use the maximal versions of the $C^\ast$-algebras, then the same construction defines an index class (resp. a local index class) in the $K$-theory of the maximal version of corresponding $C^\ast$-algebra. 
\end{remark}

\subsubsection{Higher rho invariant}\label{sec:rho}
With $M$ from above,  suppose in addition $M$ is endowed with a complete Riemannian metric $h$ whose scalar curvature $\kappa$ is uniformly positive everywhere, then the associated Dirac operator $D = D_M$ naturally defines a class in $K_1(C_{L, 0}^\ast(M)^\Gamma)$. Indeed, recall that
\[ D^2 = \nabla^\ast \nabla + \frac{\kappa}{4}, \]
where $\nabla: C^\infty(M, S) \to C^\infty(M, T^\ast M\otimes S)$ is the connection and $\nabla^\ast$ is the adjoint of $\nabla$. It follows immediately that $ D$ is invertible in this case. So we can define
\[  F = D|D|^{-1}. \]
Note that $\frac{F+1 }{2}$ is a genuine projection. Define $F(t)$ as in formula $\eqref{eq:path2}$, and denote $q(t) = \frac{F(t) + 1}{2}$. By the construction from Section $\ref{sec:kt}$, we form the path of unitaries $u(t) = e^{2\pi i q(t)}, 0\leq t < \infty$, in $ (C_L^\ast(M)^\Gamma)^+$. Notice that $u(0) = 1$. So the path $u(t), 0\leq t < \infty,$ gives rise to a class in $K_1(C_{L, 0}^\ast(M)^\Gamma)$.

\begin{definition}
The higher rho invariant of $(D, h)$ is defined to be the $K$-theory class
\[ [u(t)]\in K_1(C_{L, 0}^\ast(M)^\Gamma) \]
and will be denoted by $\rho(D, h)$ 
from now on. 
\end{definition}

The higher rho invariant was first introduced by Higson and Roe \cite[Definition 7.1]{MR2761858}. Our formulation is slightly different from that of Higson and Roe. The equivalence of the two definitions is proved in Section $\ref{sec:coin}$.

Again, the even dimensional case is essentially the same, where one needs to work with the natural $\mathbb Z/2\mathbb Z$-grading on the spinor bundle.  We leave the details to the interested reader (see also Section $\ref{sec:idem}$ below). 

%\begin{remark}
%Notice that we chose the linear path between $e^{2\pi iq(0)}$ and $1$ in the definition of the higher rho invariant above. In fact, choosing a different path of invertible elements between $e^{2\pi iq(0)}$ and $1$ may give a different $K$-theory class.  For the definiteness, we shall always choose a ``trivial'' path such as the linear path between $e^{2\pi iq(0)}$ and $1$ from now on.
%\end{remark}

\begin{remark}
If we use the maximal version of the $C^\ast$-algebra $C_{L, 0}^\ast(M)^\Gamma$, then the same construction defines a higher rho invariant in the $K$-theory of the maximal version of $C_{L, 0}^\ast(M)^\Gamma$. 
\end{remark}

\subsubsection{Kasparov $KK$-theory}\label{sec:kk}

In this subsection, we recall some standard facts from Kasparov $KK$-theory \cite{GK80}. The discussion will be very brief. We refer the reader to  \cite{BB98} for the details. 

Let $\mathcal A$ and $\mathcal B$ be graded $C^\ast$-algebras. Recall that a Kasparov module for $(\mathcal A, \mathcal B)$ is a triple $(E, \pi, \mathscr F)$, where $E$ is a countably generated graded Hilbert module over $\mathcal B$, $\pi: \mathcal A\to \mathbb B(E)$ is a graded $\ast$-homomorphism and $\mathscr F$ is an odd degree operator in $\mathbb B(E)$ such that $[\mathscr F, \pi(a)], (\mathscr F^2 - 1) \pi(a),$ and $(\mathscr F - \mathscr F^\ast) \pi(a)$ are all in $\mathbb K(E)$ for all $a\in \mathcal A$.  Here $\mathbb B(E)$ is the algebra of all adjointable $\mathcal B$-module homomorphisms on $E$ and $\mathbb K(E)$ is the algebra of all adjointable compact $\mathcal B$-module homorphisms on $E$. Of course, we have $\mathbb K(E) \subseteq \mathbb B(E)$ as an ideal. The Kasparov $KK$-group of $(\mathcal A, \mathcal B)$, denoted $KK^0(\mathcal A, \mathcal B)$, is the group generated by  all homotopy equivalence classes of Kasparov modules of $(\mathcal A, \mathcal B)$ (cf. \cite[Section 17]{BB98}). 

By definition, $KK^1(\mathcal A, \mathcal B) := KK^0(\mathcal A, \mathcal B\widehat \otimes \ccl_1)$, where $\ccl_1$ is the complexified Clifford algebra of $\mathbb R$ and $\widehat\otimes$ stands for graded tensor product. By Bott periodicity, we have a natural isomorphism
\[ \omega: KK^1(\mathcal A, \mathcal B) \xrightarrow{\ \cong\ } KK^0(C_0((0,1)) \widehat\otimes \mathcal A, \mathcal B),\]
where $C_0((0, 1))$  is trivially graded. More precisely, $\omega$ is achieved by taking  the Kasparov product of elements in  $KK^1(\mathcal A, \mathcal B)$ with a generator\footnote{This generator $\gamma$ is fixed once and for all.} $\gamma\in KK^1(C_0((0, 1)), \mathbb C) \cong \mathbb Z$. 
%\[ KK^1(\mathcal A, \mathcal B) \times KK^1(C_0((0,1)), \mathbb C) \to  KK^0(C_0((0, 1)) \widehat\otimes \mathcal A, \mathcal B). \]
In particular, the natural isomorphism $\omega$ respects the Kasparov product of $KK$-theory. We shall use $ KK^0(C_0((0, 1)) \widehat\otimes \mathcal A, \mathcal B)$ and $ KK^1(\mathcal A, \mathcal B)$ interchangeably, if no confusion arises.

Recall that, in the case where $\mathcal A = \mathbb C$ and $\mathcal B$ is \textit{trivially  graded}, there is a natural isomorphism
\begin{equation}\label{eq:k-kk}
 \vartheta: K_i(\mathcal B) \xrightarrow{\ \cong \ }  KK^i (\mathbb C ,  \mathcal B),
\end{equation}
where $\vartheta$ is defined as follows. For the moment, let $\mathcal B$ be a unital $C^\ast$-algebra. We will comment on how to deal with the nonunital case in a moment. 
\begin{enumerate}[(i)]
\item \textbf{Even case.} Let $p_i$ be projections in $M_{k_i}(\mathcal B) = M_{k_i}(\mathbb C) \otimes \mathcal B$, for $i=1, 2$. We view $\mathcal B^{\oplus k_i} = \bigoplus_{n=1}^{k_i} \mathcal B$ as a trivially graded Hilbert module over $\mathcal B$. Define $\pi_{p_i}$ to be the homomorphism from $\mathbb C$ to $\mathbb B(\mathcal B^{\oplus k_i}) = M_{k_i}(\mathcal B)$ by $\pi_{p_i}(1) = p_i$. It is clear that $(\mathcal B^{\oplus k_i}, \pi_{p_i}, 0)$
is a Kasparov module of $(\mathbb C, \mathcal B)$. Then $\vartheta$ is defined as
\begin{equation}\label{eq:ideven}
 \vartheta ([p_1] - [p_2]) = [(\mathcal B^{\oplus k_1}, \pi_{p_1}, 0)] - [(\mathcal B^{\oplus k_2}, \pi_{p_2}, 0)] \in KK^0(\mathbb C, \mathcal B).  
\end{equation}

\item \textbf{Odd case.} In this case, it is more convenient to use $KK^0(C_0((0, 1)), \mathcal B)$ in place of $KK^1(\mathbb C, \mathcal B)$. Let $u$ be a unitary in $M_k(\mathcal B)$. We again view $\mathcal B^{\oplus k}$ as a trivially graded Hilbert module over $\mathcal B$. 	 Define $\pi_u$ to be the homomorphism from $C_0((0, 1))$ to $\mathbb B(\mathcal B^{\oplus k}) = M_k(\mathcal B)$ given by $\pi_u(f) = u - 1$ where $f(s) = e^{2\pi i s} -1.$ Then $\vartheta$ is defined as 
\begin{equation}\label{eq:idodd}
 \vartheta(u) = [(\mathcal B^{\oplus k}, \pi_u, 0)] \in KK(C_0((0, 1)), \mathcal B). 
\end{equation}
\end{enumerate}

In general $\mathcal B$ is not unital. In this case, we first carry out the construction for $\widetilde{\mathcal B}$ the unitization of $\mathcal B$. Then one can easily verify that the isomorphism $\vartheta : K_i(\widetilde{\mathcal B}) \xrightarrow{\cong} KK^i(\mathbb C, \widetilde{\mathcal B})$ maps the subgroup $K_i(\mathcal B)$ to the subgroup $KK^i(\mathbb C, \mathcal B)$ isomorphically.
\[  
\xymatrix{ K_i(\mathcal B) \ar@{^{(}->}[d] \ar@{-->}[r]^-{\vartheta}  & KK^i(\mathbb C, \mathcal B) \ar@{^{(}->}[d]   \\
K_i(\widetilde{\mathcal B}) \ar[r]^-{\vartheta} & KK^i(\mathbb C, \widetilde{\mathcal B})
}
\]

%\begin{remark}
% Recall that, in $C_0((0, 1))$,   the closure of the $\ast$-subalgebra generated by the function $ f(s) = e^{2\pi i s} -1$  is precisely $C_0((0, 1))$. Therefore a $\ast$-homomorphism $\psi$ from $C_0((0, 1))$ to any unital $C^\ast$-algebra is completely determined by its value on $f$. Moreover, $\psi(f) + 1$ is a unitary.
%\end{remark}

\begin{remark}
 Note that, even when $\mathcal B$ is trivially graded, the Hilbert module $E$ in a Kasparov module $(E, \pi, \mathscr F)$ may still carry a nontrivial $\mathbb Z/2\mathbb Z$-grading. If $E$ happens to be trivially graded as well, then it implies that $\mathscr F = 0 $, since the only odd degree operator in the trivially graded algebra $\mathbb B(E)$ is $0$. 
\end{remark}

Now suppose $\mathcal B = C^\ast_L(X)^\Gamma$, where $X$ is a locally compact metric space with a proper and cocompact isometric action of $\Gamma$. Let $E_X$ be a Hilbert module over $C^\ast_L(X)^\Gamma$ and $\pi_{\mathbb C}$ be the trivial homomorphism from $\mathbb C$ to $ \mathbb B (E_X)$, i.e., $\pi_{\mathbb C}(1) = 1$.
\begin{enumerate}[(i)]
\item \textbf{Even case.} Let $(H_X, F)$ be  an even cycle in $K_0^\Gamma(X)$ and $G(t)$, $0 \leq t < \infty$, as in formula $\eqref{eq:path}$. Let $E_X = C_L^\ast(X)^\Gamma \oplus C_L^\ast(X)^\Gamma $  with its grading given by the operator   $\lambda = \left(\begin{smallmatrix} 1 & 0 \\ 0 & -1\end{smallmatrix}\right)$. Then $E_X$ is a graded Hilbert module over $C_L^\ast(X)^\Gamma$.
 Define $\mathscr F = \left(\begin{smallmatrix} 0 & G(t)^\ast \\ G(t) & 0 \end{smallmatrix}\right).$ It is easy to check that $(E_X, \pi_{\mathbb C}, \mathscr F)$ defines a Kasparov module in $KK^0 (\mathbb C , C_L^\ast(X)^\Gamma).$ 

\item \textbf{Odd case.} Let $(H_X, F)$ be  an odd cycle in $K_1^\Gamma(X)$ and $G(t)$, $0 \leq t < \infty$, as in formula $\eqref{eq:path}$. Let $E_X = C_L^\ast(X)^\Gamma \widehat \otimes \ccl_1$, viewed as a graded Hilbert module over $C_L^\ast(X)^\Gamma \widehat \otimes \ccl_1$ itself. Define $\mathscr F = G(t) \widehat \otimes e, $ where $e$ is an odd degree element in $\ccl_1$ with $e^2 =1$.  Then it is easy to check that $(E_X, \pi_{\mathbb C}, \mathscr F)$ defines a Kasparov module in $KK^1 (\mathbb C , C_L^\ast(X)^\Gamma).$ 

\end{enumerate}

In fact, the map 
\[  (H_X, F) \mapsto (E_X, \pi_{\mathbb C}, \mathscr F)   \]
induces a natural isomorphism $\nu: K_i^\Gamma(X) \xrightarrow{\ \cong \ } KK^i(\mathbb C, C^\ast_L(X)^\Gamma)$. Recall that the local index map  $\ind_L: K_i^\Gamma(X) \xrightarrow{\ \cong \ } K_i(C^\ast_L(X)^\Gamma)$ is also a natural isomorphism (cf. Section $\ref{sec:localind}$). Let  $\vartheta : K_i ( C^\ast_L(X)^\Gamma)  \xrightarrow{\ \cong \ } KK^i (\mathbb C ,  C^\ast_L(X)^\Gamma)$ be the isomorphism defined in formula $\eqref{eq:k-kk}$. Then the following diagram commutes:
\[ \xymatrix{
  & K_i^\Gamma(X) \ar[ld]_{\ind_L} \ar[rd]^\nu & \\
K_i(C^\ast_L(X)^\Gamma) \ar[rr]^{\vartheta } &   & KK^i(\mathbb C, C^\ast_L(X)^\Gamma)}
\]

\begin{example}\label{ex:dirac}
Suppose $M$ is a complete spin manifold without boundary. Assume that there is a discrete group $\Gamma$ acting on $M$  properly and cocompactly by isometries. In addition, we assume the action of $\Gamma$ preserves the spin structure on $M$. Denote the spinor bundle by $S$ and the associated Dirac operator by $D = D_M$. 
\begin{enumerate}[(1)]
\item  If $M$ is even dimensional, then there is a natural $\mathbb Z/2\mathbb Z$-grading on the spinor bundle $S$. Let us write $S= S^+ \oplus S^-$. With respect to this grading,  $D$ has odd degree. We write $D  = \left(\begin{smallmatrix} 0 & D^- \\ D^+ & 0 \end{smallmatrix}\right) $ and  $F=  \left(\begin{smallmatrix} 0 & F^- \\ F^+ & 0 \end{smallmatrix}\right)$, where $ F = D(D^2 + 1)^{-1/2}$. Note that $F^- = (F^+)^\ast$.
Roughly speaking, we in fact only use half of $F$, i.e. $F^+$,  for the construction of $\ind_L(D_M)$. Now for the construction of the corresponding Kasparov module, it is actually more natural to use the whole $F$. Let us be more specific. Recall that for the construction of the localization algebra $C_L^\ast(M)^\Gamma$, we need to fix an admissible $(M, \Gamma)$-module $H_M$. Of course, up to isomorphism, $C_L^\ast(M)^\Gamma$ does not depend  on this choice. If we choose\footnote{To be precise, we possibly need to take the direct sum of $L^2(M, S^\pm)$ with an admissible $(M, \Gamma)$-module $\mathcal H$ (cf. Remarks $\ref{rm:ad} $ and $\ref{rm:ad2}$). However, this does not affect the discussion that follows.}  $H_M = L^2(M, S^\pm)$, then we denote the resulting localization algebra by $C^\ast_L(M, S^\pm)^\Gamma$. Now let $E_{(M, S)}= C_L^\ast(M, S^+)^\Gamma \oplus C_L^\ast(M, S^-)^\Gamma \cong  C_L^\ast(M)^\Gamma\oplus C_L^\ast(M)^\Gamma$, which inherits a natural $\mathbb Z/2\mathbb Z$-grading from the spinor bundle $S$. Define $\mathscr F := F(t) = \left(\begin{smallmatrix} 0 & F(t)^- \\ F(t)^+ & 0 \end{smallmatrix}\right)$, where $F(t)$ is as in formula $\eqref{eq:path2}$. Then $ (E_{(M, S)}, \pi_{\mathbb C}, \mathscr F)$ is the corresponding Kasparov module associated to $D_M$ in $KK^0(\mathbb C, C_L^\ast(M)^\Gamma)$.
\item If $M$ is odd dimensional, then the spinor bundle $S$ is trivially graded. In this case, we simply use the general construction from before. However, for notational simplicity, we will also write $ (E_{(M, S)}, \pi_{\mathbb C}, \mathscr F)$ for the corresponding Kasparov module associated to $D_M$, even though $S$ is trivially graded. 
\item  Now suppose the metric $h$ on $M$ (of either even or odd dimension) has positive scalar curvature. In this case, we define $F = D|D|^{-1}.$ We see that $F(t)$ satisfies the condition $F(0)^2 = 1$. It follows that the Dirac operator $D_M$ (together with the metric $h$) naturally defines a Kasparove module $(E^0_{(M, S)}, \pi_{\mathbb C}, \mathscr F)$ in $KK^i(\mathbb C, C_{L, 0}^\ast(M)^\Gamma)$, where 
\begin{enumerate}[(a)]
\item in the even case, $E^0_{(M, S)} = C_{L, 0}^\ast(M)^\Gamma \oplus C_{L, 0}^\ast(M)^\Gamma$ as a $\mathbb Z/2\mathbb Z$-graded Hilbert module over $C_{L, 0}^\ast(M)^\Gamma$, with its grading given by $\lambda = \left(\begin{smallmatrix} 1 & 0 \\ 0 & -1\end{smallmatrix}\right) $,
\item in the odd case, $E^0_{(M, S)} = C_{L, 0}^\ast(M)^\Gamma\widehat \otimes \ccl_1$ as a $\mathbb Z/2\mathbb Z$-graded Hilbert module over $C_{L, 0}^\ast(M)^\Gamma\widehat \otimes \ccl_1$ itself,
\end{enumerate} 
and also  $\pi_{\mathbb C}$ and  $\mathscr F$ are defined similarly as before. We emphasis that the condition  $F(0)^2 = 1$ is essential here, since it is needed to verify that 
\[ (\mathscr F^2 - 1)\pi_{\mathbb C}(1) \in \mathbb K(E^0_{(M, S)}) = \begin{cases} M_2(\mathbb C) \otimes C_{L, 0}^\ast(M)^\Gamma & \textup{even case,}  \\   C_{L, 0}^\ast(M)^\Gamma \widehat\otimes \ccl_1 & \textup{odd case.} \end{cases} \] 

%Here $ \mathbb K(E^0_{(M, S)}) = M_2(\mathbb C) \otimes C_{L, 0}^\ast(M)^\Gamma$ in the even case, and $\mathbb K(E^0_{(M, S)}) = C_{L, 0}^\ast(M)^\Gamma \widehat\otimes \ccl_1$ in the odd case. 
\end{enumerate}

\end{example}

We conclude this subsection by the following observation, which will be used  in the proof of our main theorem (Theorem $\ref{thm:main}$). Intuitively speaking, the observation is  that  Kasparov $KK$-theory allows us to represent  $K$-theory classes, which carry ``spatial'' information, by operators, which carry ``spectral'' information. 

Let $M_1$ (resp. $M_2$) be a complete spin manifold without boundary  equipped with a proper and cocompact isometric action of $\Gamma_1$ (resp. $\Gamma_2$). Assume $\Gamma_1$ (resp. $\Gamma_2$) preserves the spin structure of $M_1$ (resp. $M_2$). Let 
\[ \vartheta_1 : K_i(C^\ast_L(M_1)^{\Gamma_1}) \to KK^i (\mathbb C, C_L^\ast(M_1)^{\Gamma_1}) \]
\[ \vartheta_2 : K_i(C^\ast_L(M_2)^{\Gamma_1}) \to KK^i (\mathbb C, C_L^\ast(M_2)^{\Gamma_1}) \]
\[ \vartheta_3 : K_i(C^\ast_L(M_1)^{\Gamma_1}\otimes C_{L}^\ast(M_2)^{\Gamma_2}) \to KK^i(\mathbb C, C_L^\ast(M_1)^{\Gamma_1} \otimes C_{L}^\ast(M_2)^{\Gamma_2} ) \]
be natural isomorphisms defined as in formula $\eqref{eq:k-kk}$. Then it follows from the standard construction of the Kasparov product (cf. \cite[Section 18]{BB98}) that we have the following commutative diagram:
\[ 
 \xymatrix{  K_i(C^\ast_L(M_1)^{\Gamma_1}) \times K_j(C_{L}^\ast( M_2)^{\Gamma_2})  \ar[r]^{\otimes_K} \ar[d]_{\cong}^{\vartheta_1\times \vartheta_2} & K_{i+j}(C^\ast_L(M_1)^{\Gamma_1}\otimes C_{L}^\ast(M_2)^{\Gamma_2}) \ar[d]_{\cong}^{\vartheta_3}  \\
 KK^i (\mathbb C, C_L^\ast(M_1)^{\Gamma_1}) \times KK^j(\mathbb C, C_{L}^\ast(M_2)^{\Gamma_2})) \ar[r]^-{\otimes_{KK}} & KK^{i+j}(\mathbb C, C_L^\ast(M_1)^{\Gamma_1} \otimes C_{L}^\ast(M_2)^{\Gamma_2} ) 
} \]
where $\otimes_K$ is the standard external product in $K$-theory and $\otimes_{KK}$ is the Kasparov product in $KK$-theory. In other words, the commutative diagram states that the Kasparov product is compatible with the external $K$-theory product.

Now let $D_1 = D_{M_1}$, $D_2 = D_{M_2}$ and $D_3 = D_{M_1\times M_2}$ be the associated Dirac operator on $M_1$, $M_2$ and $M_1\times M_2$ respectively.  Let us write 
\[  \beta_1 = \ind_L(D_1) \in K_i(C^\ast_L(M_1)^{\Gamma_1}), \quad  \beta_2 = \ind_L(D_2) \in K_j(C^\ast_L(M_2)^{\Gamma_2})\]
\[  \textup{and} \quad \beta_3 = \ind_L(D_3) \in  K_{i+j}(C^\ast_L(M_1 \times M_2)^{\Gamma_1\times \Gamma_2}). \]
Notice that there is a natural homomorphism 
\[ \Psi: C^\ast_L(M_1)^{\Gamma_1}\otimes C_{L}^\ast(M_2)^{\Gamma_2} \to  C^\ast_L(M_1\times M_2)^{\Gamma_1\times \Gamma_2}.\]
Moreover, $\Psi$ induces an isomorphism on $K$-theory
\[ \Psi_\ast: K_{i+j}(C^\ast_L(M_1)^{\Gamma_1}\otimes C_{L}^\ast(M_2)^{\Gamma_2}) \xrightarrow{\cong}  K_{i+j}(C^\ast_L(M_1\times M_2)^{\Gamma_1\times \Gamma_2}).\] 
Then we have the following claim.
\begin{claim}\label{claim:local}
$ \Psi_\ast(\beta_1 \otimes_K \beta_2) = \beta_3  $ in $K_{i+j}(C^\ast_L(M_1 \times M_2)^{\Gamma_1\times \Gamma_2})$.
\end{claim}
\begin{proof}
 Denote $\alpha_k = \vartheta_k (\beta_k)$, for $k\in \{1, 2, 3\}$ (cf. formulas $\eqref{eq:ideven}$ and $\eqref{eq:idodd}$). On one hand,  by the commutative diagram above, we have
\[  \alpha_1 \otimes_{KK} \alpha_2 = \vartheta_3(\beta_1\otimes_K \beta_2). \] 
On the other hand, within the same $KK$-theory class, the classes $\alpha_1$,  $\alpha_2$ and $\alpha_3$ can be represented respectively by $\alpha'_1 = (E_{(M_1, S)}, \pi_{\mathbb C}, \mathscr F_1)$,   $\alpha'_2 = (E_{(M_2, S)}, \pi_{\mathbb C}, \mathscr F_2)$ and $\alpha'_3 = (E_{(M_1\times M_2, S)}, \pi_{\mathbb C}, \mathscr F_3)$  as in Example $\ref{ex:dirac}$, where $\mathscr F_k$ is constructed out of the Dirac operator $D_k$. Intuitively speaking, the key idea is that, in the Kasparov $KK$-theory framework, these $K$-theory classes can be represented by the Dirac operators at the operator level. 

 Now one standard way to define the Kasparov product of $\alpha'_1$ and $\alpha'_2$ is to use connections (cf. \cite[Chapter 18]{BB98}). The notion of connection was due to Connes and Skandalis \cite{MR775126}. The Kasparov product of $\alpha'_1$ and $\alpha'_2$ is defined to be any Kasparov module\footnote{Different candidates of Kasparov module will give the same $KK$-theory class, as long as all the necessary conditions are satisfied.} which satisfies several standard conditions (cf. \cite[Definition 18.3.1 \& Definition 18.4.1]{BB98}). In particular, once we have a candidate for the Kasparov product, then it only remains to see whether these standard conditions are satisfied by this candidate. In our case, it is not difficult to verify that the candidate $\alpha'_3$ indeed satisfies all these standard conditions. Therefore, we have 
\[ \alpha'_1\otimes_{KK} \alpha'_2 = \alpha'_3. \]
An alternative way is to use unbounded Kasparov modules  \cite{MR715325}. In any case,  we have
\[   \vartheta_3(\beta_1\otimes_K \beta_2)  = \alpha_1 \otimes_{KK} \alpha_2 = \alpha'_1 \otimes_{KK} \alpha'_2 = \alpha'_3 =  \alpha_3 = \vartheta_3(\beta_3). \]
So $ \beta_1 \otimes_K \beta_2 = \beta_3$.
This proves the claim.
\end{proof}
Now if one of the manifolds, say $M_2$, is equipped with a metric $h$ of positive scalar curvature, then we can replace $C_L^\ast(M_2)^{\Gamma_1}$  by $C_{L,0}^\ast(M_2)^{\Gamma_1}$. More precisely, let us denote
\[ \xi_1 = \ind_L(D_1)\in K_i(C_{L}^\ast(M_1)^{\Gamma_1}),\quad \xi_2 = \rho(D_2, h)  \in K_j(C_{L, 0}^\ast(M_2)^{\Gamma_2}) \]
\[\textup{and} \quad  \xi_3 = \rho(D_3, h)\in K_{i+j}(C_{L,0}^\ast(M_1\times M_2)^{\Gamma_1\times \Gamma_2}).\]
By construction, we have $\xi_1\otimes_K \xi_2 \in K_{i+j}(C_{L}^\ast(M_1)^{\Gamma_1}\otimes C_{L, 0}^\ast(M_2)^{\Gamma_2}).$ Let $\iota$ be the natural homomorphism
\[ \iota: C_{L}^\ast(M_1)^{\Gamma_1}\otimes C_{L, 0}^\ast(M_2)^{\Gamma_2} \to C_{L,0}^\ast(M_1\times M_2)^{\Gamma_1\times \Gamma_2}. \]
Then the following claim holds. 
\begin{claim}\label{claim:rho}
 $\iota_\ast(\xi_1\otimes_K \xi_2) = \xi_3$ in $K_{i+j}(C_{L,0}^\ast(M_1\times M_2)^{\Gamma_1\times \Gamma_2})$.
\end{claim} 
\begin{proof}
The proof works exactly the same as that of Claim $\ref{claim:local}$. The key idea again is that, in the Kasparov $KK$-theory framework, these $K$-theory classes can be represented by the Dirac operators at the operator level. For each $\xi_i$,  let $\xi_i'$ be the corresponding Kasparov module defined as in Example $\ref{ex:dirac}$. It suffices to verify that $\xi'_3$ is the Kasparov product of $\xi'_1$ and $\xi'_2$. Again $\xi'_3$ is a natural candidate for the Kasparov product of $\xi'_1$ and $\xi'_2$. Once we have a candidate, all it remains is to verify that the standard conditions as those in \cite[Definition 18.3.1 \& Definition 18.4.1]{BB98} are satisfied by $\xi'_3$. Indeed, it is not difficult to verify these conditions for $\xi'_3$ in our case. Therefore, $\xi'_3$ is the Kasparov product of $\xi'_1$ and $\xi'_2$. This finishes the proof. 
\end{proof}
 The case where both $M_1$ and $M_2$ are equipped with metrics of positive scalar curvature is also similar. We omit the details.

\section{Manifolds with positive scalar curvature on the boundary}\label{sec:mbpsc}
In this section, we discuss the higher index classes of  Dirac operators on spin manifolds with boundary, where the boundary is endowed with a positive scalar curvature metric. 

Throughout the section, let $M$ be a complete spin manifold with boundary $\partial M$ such that
\begin{enumerate}[(i)]
\item the metric on $M$ has product structure near $\partial M$ and its restriction to $\partial M$ has positive scalar curvature;
\item there is an proper and cocompact isometric action of a discrete group $\Gamma$ on $M$;
\item the action of $\Gamma$ preserves the spin structure of $M$. 
\end{enumerate}  

We attach an infinite cylinder $\mathbb R_{\geq 0} \times \partial M = [0, \infty)\times \partial M$ to $M$. If we denote the Riemannian metric on $\partial M$ by $h$, then we endow $\mathbb R_{\geq 0}\times \partial M$ with the standard product metric $dr^2 + h$. Notice that  all relevant geometric structures on $M$ extend naturally  to $M_\infty = M\cup_{\partial M}( \mathbb R_{\geq 0}\times \partial M)$. The action of $\Gamma$ on $M$ also extends to  $M_\infty$. 

Let us set $M_{(n)} = M\cup_{\partial M} ([0, n]\times \partial M)$, for $n\geq 0$. In particular, $M_{(0)} = M$. Again,  the action of $\Gamma$ on $M$ extends naturally to  $M_{(n)}$, where $\Gamma$ acts on $[0, n]$ trivially.      

\begin{remark}
In fact, without loss of generality, we assume that we have fixed an identification of a neighborhood of  $\partial M$ in $M$ with $[-3, 0]\times \partial M$.  For $ 0 \leq k \leq 3$, we will write 
\[  M_{(-k)} = M\backslash ((-k, 0]\times \partial M). \]
This will be used later for notational simplification. See Figure $\ref{fig:coord}$ below.
\end{remark}

\begin{figure}[h]

\hspace*{3.5cm}
\begin{tikzpicture}
\draw    (8.5, 1.5) -- (4, 1.5);
\draw (4, 1.5).. controls (3, 1.5) and (3, 1) .. (2, 1);
\draw (2, 1).. controls  (1, 1) and (0.5, 1.6) .. (0.5, 2);
\draw (0.5, 2).. controls (0.5, 2.7) and (1, 3.15).. (1.75, 3.15);
\draw (1.75, 3.15).. controls (2.5, 3.15) and (2.5, 2.5) .. (3.5, 2.5);
\draw (1.08, 2) .. controls (1.65, 2.5) and  (2, 2.5) .. (2.6, 1.9);
\draw (3.5, 2.5) -- (8.5, 2.5);
\draw (1.35, 2.2) .. controls (1.78, 1.9) .. (2.35, 2.07);
\draw (4, 1.4) -- (4, 1.6);
\draw [dashed](4, 1.5) .. controls (4.4, 1.5) and (4.4, 2.5) .. (4, 2.5);
\draw (4, 1.5) .. controls (3.6, 1.5) and  (3.6, 2.5) .. (4, 2.5);
\draw (6, 1.4) -- (6, 1.6);
\draw [dashed](6, 1.5) .. controls (6.4, 1.5) and  (6.4, 2.5) ..  (6, 2.5);
\draw (6, 1.5) .. controls (5.6, 1.5) and  (5.6, 2.5) .. (6, 2.5);
\node [below] at (4, 1.3){$-3$};
\node  [below] at (6, 1.3){$0$};
\node [above] at (6,  2.6){$\partial M$};
\node [above] at (4.5, 3){$M$};
\node [right] at (7, 1){$\mathbb R_{\geq 0}\times \partial M$};
\end{tikzpicture}

\caption{Attaching $\mathbb R_{\geq 0}\times \partial M$ to $M$. }
\label{fig:coord}
\end{figure}
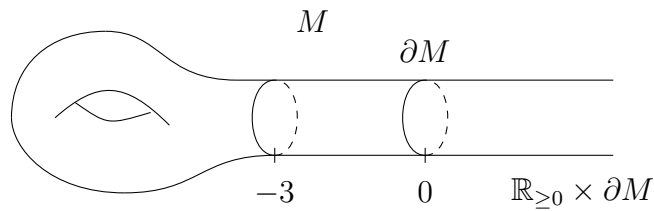

Now let $D = D_{M_\infty}$ be the associated Dirac operator on $M_\infty= M\cup_{\partial M}( \mathbb R_{\geq 0}\times \partial M)$. 
\begin{claim}\label{claim:fin}
With the above notation,  $D_{M_\infty}$ in fact defines an index class, denoted by  $\ind(D_{M})$, in $ K_i(C^\ast(M)^\Gamma)$.
\end{claim}
In this section, we carry out a detailed construction to verify this claim. We point out that the claim, at least in the case when there is no $\Gamma$-action, was due to Roe \cite[Proposition 3.11]{MR1399087}. A detailed proof was later given by Roe in \cite{Roe:2012kq}. Our construction is different from that of Roe. Most importantly, we construct a representative of the index class $\ind(D_{M})$ for each $t\in [0, \infty)$, with specific control of the propagation away from the boundary of $M$. This will be one of the key ingredients used in the proof of our main result, Theorem $\ref{thm:main}$, below.

\begin{remark}
The action of $\Gamma$ on $M_\infty= M\cup_{\partial M}( \mathbb R_{\geq 0}\times \partial M)$ is not cocompact, that is, $M_\infty/\Gamma$ is not compact. 
\end{remark}
\begin{remark}
The above claim is not true in general, if we drop the positive scalar curvature assumption near the boundary,.
\end{remark}
Let $\kappa$ be the scalar curvature function on $M_\infty = M\cup_{\partial M}( \mathbb R_{\geq 0}\times \partial M)$ and $\rho$ be  a $\Gamma$-invariant nonnegative smooth function on $M_\infty$ satisfying:  
\begin{enumerate}[(a)]
\item $\supp(\rho)\subseteq M_{(n)} = M\cup_{\partial M}([0, n]\times \partial M)$ for some $n\in \mathbb N$;
\item there exists $c > 0$ such that $ \rho(y)+\frac{\kappa(y)}{4} > c $ for all $y\in M_\infty$. 
\end{enumerate}
Define
\[  F = F_\rho =  \frac{D}{\sqrt {D^2 + \rho} }. \]
Recall that \[ \frac{1}{\sqrt x} = \frac{2}{\pi}\int_0^\infty \frac{1}{x + \lambda^2}  d\lambda. \]
It follows that  
\[  {\frac{1}{\sqrt{D^2 + \rho}} }= \frac{2}{\pi}\int_0^\infty (D^2 + \rho + \lambda^2)^{-1} d\lambda. \]
Now by using the equality
\begin{align*}
 &(D^2 + \rho+ \lambda^2)^{-1} D - D (D^2 + \rho + \lambda^2)^{-1} = (D^2 + \rho + \lambda^2)^{-1}  [D, \rho ] (D^2 + \rho + \lambda^2)^{-1},
\end{align*}
we have
\begin{align*}
F^2 & = D(D^2 + \rho)^{-1/2}D(D^2 + \rho)^{-1/2} \\
& = D \left(\frac{2}{\pi}\int_0^\infty (D^2 + \rho + \lambda^2)^{-1} d\lambda \right) D (D^2 + \rho)^{-1/2}\\
& = \left( D^2 \cdot \frac{2}{\pi} \int_0^\infty (D^2 + \rho+ \lambda^2)^{-1} d\lambda +  D \cdot \frac{2}{\pi} \int_0^\infty R(\lambda) d\lambda \right) (D^2 + \rho)^{-1/2} \\
& = D^2 (D^2 + \rho)^{-1} + D R (D^2 + \rho)^{-1/2}	\\
& = 1 - \rho (D^2 + \rho)^{-1} + D R (D^2 + \rho)^{-1/2}
\end{align*} 
where $R = \frac{2}{\pi}\int_0^\infty R(\lambda) d\lambda$ with 
\[ R(\lambda) =(D^2 + \rho + \lambda^2)^{-1}  [D, \rho ] (D^2 + \rho + \lambda^2)^{-1}.\]
Since  $\supp(\rho)$ and $\supp([D, \rho])$ are $\Gamma$-cocompact, it follows that both   $\rho (D^2 + \rho)^{-1} $ and $D R (D^2 + \rho)^{-1/2}$ are in $C^\ast(M_\infty)^\Gamma$.

In fact, we can choose $\rho$ such that the following are satisfied. 
\begin{enumerate}[(i)]
\item $\rho \equiv C $ on $M_{(n_1)} = M\cup_{\partial M}([0,n_1]\times \partial M)$ and $\rho \equiv 0 $ on $[n_2, \infty)\times \partial M$ for some $ n_2 > n_1 \geq 0$ and some constant $C>0$.
\item $\rho(x, y) = \rho(x, y')$ for all $y, y'\in \partial M$, where $(x, y) \in (n_1, n_2)\times \partial M$. In other words, $\rho$ is constant along $\partial M$. In particular, it follows that  
\[  [D , \rho] = \rho'\cdot c(\partial_x) \]
%\[  [D^2, \rho] = -\rho''  - 2 \rho' \partial_x   \]
where $c(\partial_x)$ is the Clifford multiplication of $\partial_x$.
% $\rho' = \partial_x \rho$ and $\rho'' = \partial^2_x \rho$.
\item %$\| (D^2 + \rho)^{-1} [D^2, \rho] (D^2 + \rho)^{-1}\| < \varepsilon $ and
 $ \|D R(D^2 + \rho)^{-1/2}\| < \frac{1}{100\cdot \omega_0}$, by choosing $\rho$ so that $\rho'$ has small supremum-norm.  Here $\omega_0 =  2\pi (\|F\|+1)^2 \cdot e^{2\pi (\|F\|+1)} $. The reason for choosing such a constant $\omega_0$ will become clear in Definition $\ref{def:inv}$.
\end{enumerate}

For any number $c_0 >0 $ such that $\supp(\rho) \subset M_{(c_0)}$, let us decompose the space $L^2(M_\infty, S)$ into a direct sum $L^2(M_{(c_0)}, S) \oplus L^2(\mathbb R_{\geq c_0}\times \partial M, S)$.  
We write 
\[ F^2= \begin{pmatrix} A_{11} & A_{12} \\ A_{21}  & A_{22} \end{pmatrix}\]
with respect to this decomposition. Notice that $A_{12} = A_{21}^\ast$, since $F^2$ is selfadjoint. Now by our assumption $\|D R(D^2 + \rho)^{-1/2}\| < \frac{1}{100\cdot \omega_0}$ and the formula 
\[ F^2 =  1 - \rho (D^2 + \rho)^{-1} + D R (D^2 + \rho)^{-1/2}, \]
 it is not difficult to verify that  
\begin{equation}\label{eq:estf}
 \| A_{12}\| = \|A_{21}\|< \frac{1}{100\cdot \omega_0}\quad   \textup{and} \quad  \|A_{22} - 1\|<\frac{1}{100\cdot \omega_0}. 
\end{equation} 
We point out that  the above estimates hold independent of the choice of $c_0$, as long as we have $\supp(\rho) \subset M_{(c_0)}$.

Similarly, suppose $\psi$ is a function with supported on $(0, \infty) \times \partial M$ such that $\psi$ is constant along $\partial M$. Notice that 
\begin{equation}\label{eq:comest}
[F, \psi] =  \frac{2}{\pi}\int_0^\infty (D^2 + \rho + \lambda^2)^{-1} [D^2, \psi]  (D^2 + \rho + \lambda^2)^{-1} d\lambda
\end{equation}
and $ [D^2, \psi] = -\psi''  - 2 \psi' \partial_x$, where $\psi' = \partial_x \psi$ and $\psi''= \partial_x^2\psi$.
In particular, we see that $\| [F, \varphi]\|$ is proportional to the supremum-norm of  $\varphi'$ and $\varphi''$.

Notice that although $F$ itself generally does not have finite propagation, we can choose a $\Gamma$-invariant locally finite open cover $\{U_j\}$ of $M_\infty$ and a $\Gamma$-invariant partition of unity $\{\varphi_{j}\}$ subordinate to $\{U_{j}\}$ such that  
\begin{equation*}
 G = \sum_{j}\varphi_{j}^{1/2} F \varphi_{j}^{1/2} 
\end{equation*}
has finite propagation. Moreover, by use of Equation $\eqref{eq:comest}$, for $\forall \varepsilon >0 $, we can choose the  open cover and the partition of unity appropriately so that  $\|F - G\| < \varepsilon$. \emph{Hence without loss of generality, let us assume that $F$ has finite propagation.}

\subsection{Invertibles}\label{sec:invertible}
Now suppose the dimension of $M_\infty = M\cup_{\partial M}( \mathbb R_{\geq 0}\times \partial M)$ is odd. In this subsection,  for each $t\in [0, \infty)$, we will construct $F(t)$ from $F$ with specific prescribed propagation properties.

First for  $\forall n\in \mathbb N$, we can choose a $\Gamma$-invariant locally finite open cover $\{U_{n,j}\}$ of $M_\infty$ and a $\Gamma$-invariant partition of unity $\{\varphi_{n,j}\}$ subordinate to $\{U_{n,j}\}$ with the following properties.  
\begin{enumerate}[(a)]
\item If $U_{n, j} \subset M_{(0)} = M $, then $ \diam(U_{n, j}) < \frac{1}{n}$.
\item There are precisely two open sets $U_{n,j}$ such that $U_{n, j}\cap (\mathbb R_{\geq 0}\times \partial M)$ is nonempty. For convenience, we also denote them by $W_1 = (\frac{-1}{n}, c_1)\times \partial M $ and $W_2 = (0, \infty)\times \partial M$. Here $c_1 > 0 $ is a sufficiently large number such that $\supp(\rho)\subseteq  M_{(c_0)}$.  In particular, notice that the choice of $W_2$ is independent of $n$.
\item If $\varphi_{n, j} = \varphi$ is the smooth function subordinate to the open set $W_2$, then $\varphi \equiv 1  $ on $\mathbb R_{\geq c_1}\times \partial M$ and $ \|[F, \varphi^{1/2}]\| <\frac{1}{200  \cdot \omega_0} $. Recall that $\omega_0 =  2\pi (\|F\|+1)^2 \cdot e^{2\pi (\|F\|+1)} $.
\end{enumerate}
\emph{We fix the number $c_1$ and this particular function $\varphi$ from now on}. With the above notation, we define 
\begin{equation}\label{eq:fp}
 F_t = \sum_{j} (1 - (t-n)) \varphi_{n, j}^{1/2} F \varphi_{n,j}^{1/2} + (t-n) \varphi_{n+1, j}^{1/2} F \varphi_{n+1, j}^{1/2}
\end{equation}
for $t\in [n, n+1]$. If we define $\sn_t = \frac{1}{n}$ for $t\in [n, n+1)$, then by construction the propagation of $F_t$ restricted to $M_{(-\sn_t)} = M \backslash ((-\sn_t, 0]\times \partial M) $ is bounded by $\sn_t$, hence  goes to $0$, as $t\to \infty$.  

\begin{lemma}[{cf. \cite[Lemma 2.6]{GY98}}]  We have $\|F_t\| \leq 2\|F\|$ for all $t\in [0, \infty)$. 
\end{lemma}
\begin{proof}
Notice that $F$ is selfadjoint. Consider the spectral decomposition of $F$ and define $T_1 = \chi_{[0, \infty)}(F)$ and $T_2 = \chi_{(-\infty, 0)}(F)$, where for example $\chi_{[0, \infty)}$ is the characteristic function on the interval $[0, \infty)$. Clearly, we have $\|T_i\|\leq \|F\|$ and  $ F =  T_1 - T_2$.   Hence it suffices to show that 
\[  \Big\| \sum_{j} \varphi_{n, j}^{1/2} T_i \varphi_{n,j}^{1/2} \Big\| \leq \|T_i\|. \]
Notice that $T_i\geq 0$. Therefore, we have 
\begin{align*}
&\Big\langle \Big(\sum_{j} \varphi_{n, j}^{1/2} T_i \varphi_{n,j}^{1/2}\Big) f, f\Big\rangle = \sum_j \langle \varphi_{n, j}^{1/2} T_i \varphi_{n,j}^{1/2} f, f\rangle = \sum_{j} \langle  T_i \varphi_{n,j}^{1/2} f, \varphi_{n, j}^{1/2}f\rangle  \\
&\quad \quad \leq \sum_{j} \|T_i\| \langle  \varphi_{n,j}^{1/2} f, \varphi_{n, j}^{1/2}f\rangle =  \|T_i\| \Big(\sum_{j} \langle   \varphi_{n,j}f, f\rangle\Big)  = \|T_i\| \|f\|^2
\end{align*}
for all $f\in L_2(M_\infty, S)$. This finishes the proof. 
\end{proof}

Recall that,  without loss of generality, we can assume that $F$ has finite propagation (see the discussion right before this subsection). Suppose the propagation of $F$ is $c_2$. Let $C = c_1 + c_2$, where $c_1$ is the fixed real number as above. We define 
\[  H_1 =  L^2(M_{(C)}, S) \quad  \textup{and} \quad H_2 =  L^2(\mathbb R_{\geq C}\times \partial M, S).\]
 Let us write 
\[ F_t^2= \begin{pmatrix} B_{11}(t) & B_{12}(t) \\ B_{21}(t)  & B_{22}(t) \end{pmatrix}\]
with respect to the decomposition $ L^2(M_\infty, S) = H_1 \oplus H_2$. 

\begin{lemma}\label{lm:est} We have  
\[ \|B_{12}(t)\| = \|B_{21}(t)\|< \frac{(\|F\|+1)}{100 \cdot \omega_0} \quad \textup{and} \quad \|B_{22}(t) - 1 \|< \frac{(\|F\|+1)}{100\cdot \omega_0}\] 
for all $t\in [0, \infty)$.
\end{lemma}
\begin{proof}
Recall that $\varphi$ is the smooth  function defined above  such that $\varphi \equiv 1  $ on $\mathbb R_{\geq c_1}\times \partial M$ and $ \|[F, \varphi^{1/2}]\| <\frac{1}{200 \cdot \omega_0} $. It follows that 
\[  F_t(f) = (\varphi^{1/2} F \varphi^{1/2})(f) \]
for all $f\in L^2(\mathbb R_{\geq c_1}\times \partial M)$. Therefore we have
\[  \|(F_t- F)(f)\| = \|[F, \varphi^{1/2}](f)\| < \frac{1}{200\cdot \omega_0} \|f\| \]
for all $f\in L^2(\mathbb R_{\geq c_1}\times \partial M)$. Now notice that for all $h\in H_2 = L^2(\mathbb R_{\geq C}\times \partial M)$, the support of $F(h)$ is contained in $\mathbb R_{\geq c_1}\times \partial M $. Therefore it follows that 
\begin{align*}
\| (F_t^2 - F^2)(h) \| & \leq  \| (F_t - F) F(h)\| +  \| F(F_t -F)(h) \| \\
& \leq \frac{2}{200\cdot \omega_0} \|F(h)\| \leq \frac{\|F\|}{100\cdot \omega_0} \|h\| 
\end{align*}   
for all $h\in H_2$. 
Suppose we still write 
\[  F^2 = \begin{pmatrix} A_{11} & A_{12} \\ A_{21}  & A_{22} \end{pmatrix}\]
with respect to the decomposition  $ L^2(M_\infty, S) = H_1 \oplus H_2$. Then we have 
\[   \|A_{12} - B_{12}(t)\| =  \|A_{21} - B_{21}(t)\| <\frac{\|F\|}{100\cdot \omega_0}  \quad \textup{and} \quad \|A_{22} - B_{22}(t)\| < \frac{\|F\|}{100\cdot \omega_0}.  \]
Now the lemma follows from our estimates for $\|A_{ij}\|$ in $\eqref{eq:estf}$. 
\end{proof}
Now let us define  
\[ P_t =  \frac{F_t + 1}{2}.\] 
Clearly we have 
\[  P_t^2 - P_t = \frac{F_t^2 - 1}{4} = \frac{1}{4}\begin{pmatrix} B_{11}(t)- 1 & B_{12}(t) \\ B_{21}(t)  & B_{22}(t)-1 \end{pmatrix}.\] 
If we define 
\[ E_t = \frac{1}{4}\begin{pmatrix} B_{11}(t)- 1 & 0 \\ 0  & 0 \end{pmatrix},\]
then Lemma $\ref{lm:est}$ above implies that 
\[ \| E_t - \left(P_t^2 -P_t\right) \| < \frac{(\|F\|+1)}{400 \cdot \omega_0}. \]

Recall that $e^{2\pi i P_t}$ is a representative of the index class $\ind(D_{M_\infty}) \in K_1(C^\ast( M_\infty)^\Gamma)$, for each fixed $ t\in [0, \infty)$. Let us approximate the function $e^{2\pi ix}$ by polynomials
\[f_N(x) = \sum_{n=0}^{N} \frac{(2\pi i)^n}{n!} x^n. \] 
Notice that $\|P_t\| \leq \frac{2\|F\| +1 }{2} =  \|F\| + \frac{1}{2}$ for all $t\in [0, \infty)$, and $ \sum_{n=1}^N\frac{(2\pi i)^n}{n!} \to 0$,  as $ N\to \infty$. Therefore, we can choose a sufficiently large positive integer $N$ such that 
\begin{enumerate}[(i)]
\item $ \|f_N(P_t) - e^{2\pi i P_t} \| <\frac{1}{300}$, for all $t\in [0, \infty)$;
\item $|\sum_{n=1}^N\frac{(2\pi i)^n}{n!}| < \frac{1}{300\cdot (\|F\|+1)}$.
\end{enumerate}
 In particular, it follows that $f_N(P_t)$ is invertible. We fix this integer $N$ from now on.
  
Notice that we have 
\[  P_t^n  - P_t = \left(\sum_{j= 0}^{n-2} P_t^{j}\right) (P_t^2- P_t), \]
for all $n\geq 2$.
It follows that 
\begin{align*}
f_N(P_t) &=  \sum_{n=0}^{N} \frac{(2\pi i)^n}{n!}P_t^n \\
 & = 1 + \left(\sum_{n=1}^N\frac{(2\pi i)^n}{n!}\right) P_t + \sum_{n = 1}^{N} \frac{(2\pi i)^n}{n!} (P_t^n - P_t)\\
 & = 1 + \left(\sum_{n=1}^N\frac{(2\pi i)^n}{n!}\right) P_t +  \left(\sum_{n = 1}^{N} \sum_{j= 0}^{n-2} \frac{(2\pi i)^n}{n!}P_t^{j}\right) (P_t^2- P_t)
\end{align*}
\begin{definition}\label{def:inv}
We define 
\[  u_t =  1  +  \left(\sum_{n = 1}^{N} \sum_{j= 0}^{n-2} \frac{(2\pi i)^n}{n!}P_t^{j}\right)E_t.\]
\end{definition}

 Recall that $\|P_t\| \leq \|F\|+\frac{1}{2}$ and $\omega_0 = 2\pi (\|F\|+1)^2 \cdot e^{2\pi (\|F\|+1)}$. A routine calculation shows that 
\[  \Big\|\sum_{n = 1}^{N} \sum_{j= 1}^{n-2} \frac{(2\pi i)^n}{n!}P_t^{j}\Big\| \leq 2\pi (\|F\|+1) \cdot e^{2\pi (\|F\|+1)} . \]
It follows that 
\begin{align*}
\| u_t - f_N(P_t)\| & \leq \Big\|\left(\sum_{n=1}^N\frac{(2\pi i)^n}{n!}\right) P_t\Big\|  + \Big\| \left(\sum_{n = 1}^{N} \sum_{j= 0}^{n-2} \frac{(2\pi i)^n}{n!}P_t^{j}\right) \big(E_t - (P_t^2- P_t)\big)\Big\| \\
&  <  \frac{1}{300} + 2\pi (\|F\|+1) \cdot e^{2\pi (\|F\|+1)} \frac{(\|F\| + 1)}{400\cdot \omega_0} < \frac{2}{300}.
\end{align*}
Therefore, we have
\[ \|u_t - e^{2\pi i P_t}\| \leq \|u_t - f_N(P_t)\| + \|f_N(P_t) - e^{2\pi i P_t}\| < \frac{1}{100}. \]
In particular, we see that $u_t$ is invertible. 

Notice that by construction the propagation of $P_t$ is smaller than the propagation of $F$ for all $t\in [0, \infty)$. Recall that the propagation of $F$ is $c_2$. So the propagation of $P_t$ is uniformly bounded by $c_2$ for all $t \in [0,\infty)$. Define $\Lambda= C + N\cdot c_2$. Recall that 
\[ E_t = \frac{1}{4}\begin{pmatrix} B_{11}(t)- 1 & 0 \\ 0  & 0 \end{pmatrix}\] 
with respect to the decomposition $ L^2(M_{C}, S) \oplus L^2(\mathbb R_{\geq C}\times \partial M, S)$.

\begin{lemma}
For all $t\in [0, \infty)$, $u_t$ preserves the decomposition $L^2(M_{(\Lambda)}, S)\oplus L^2(\mathbb R_{\geq \Lambda}\times \partial M, S)$. That is, if we write 
\[  u_t = \begin{pmatrix} (u_t)_{11}  & (u_t)_{12} \\
(u_t)_{21} & (u_t)_{22}
\end{pmatrix} \]
with respect to the decomposition $L^2(M_{(\Lambda)}, S)\oplus L^2(\mathbb R_{\geq \Lambda}\times \partial M, S)$, 
then 
\[ (u_t)_{12} = (u_t)_{21} = 0.\]
Moreover, $(u_t)_{22} = 1$.  
\end{lemma}
\begin{proof}
Let us further decompose $L^2(M_{(\Lambda)}, S)$ into $L^2(M_{(C)}, S) \oplus L^2([C, \Lambda]\times\partial M, S) $. If $f\in  L^2([C, \Lambda]\times \partial M, S)$, then $E_t(f) = 0$. It follows immediately that for all $h\in  L^2(M_{(\Lambda)}, S)$, the support of $E_t(h)$ is contained in $M_{(C)}$. Now recall that 
\[  u_t =  1  +  \left(\sum_{n = 1}^{N} \sum_{j= 0}^{n-2} \frac{(2\pi i)^n}{n!}P_t^{j}\right)E_t\]
and $\Lambda = C + N\cdot c_2$. Since the propagation of $P_t$ is bounded by $c_2$, we see that the support of $u_t(h)$ is contained in $M_{(\Lambda)}$ for all $h\in  L^2(M_{(\Lambda)}, S)$. Hence $(u_t)_{21} = 0$. 

Now if $f\in L^2(\mathbb R_{\geq \Lambda}\times\partial M, S) $, then $E_t(f) = 0$. It follows immediately that $(u_t)_{12} = 0$ and $(u_t)_{22} = 1$.
\end{proof}
 
So we see that, by restricting $u_t$ to $L^2(M_{(\Lambda)}, S)$, $u_t$ naturally gives rise to an element in $ (C^\ast( M_{(\Lambda)})^\Gamma)^+$, which we will still denote by $u_t$. Moreover, $M_{(\Lambda)}$ is clearly $\Gamma$-equivariant coarsely equivalent to $M_{(0)} = M$ by mapping $M_{(0)}\subseteq M_{(\Lambda)} $ identically to $M$ and projecting $[0, \Lambda]\times \partial M$ to $\{0\}\times \partial M$. This coarse equivalence map induces a (noncanonical) homomorphism  $\Phi: C^\ast( M_{(\Lambda)})^\Gamma \to C^\ast( M)^\Gamma$, which induces a canonical isomorphism  (cf. \cite{MR1219916})
\[  \Phi_\ast: K_i(C^\ast( M_{(\Lambda)})^\Gamma)\xrightarrow{\cong} K_i(C^\ast( M)^\Gamma).  \] 
If no confusion arises, we also denote the image of $u_t$ under the homomorphism $\Phi$ by $u_t \in (C^\ast( M)^\Gamma)^+$. 

\begin{definition}\label{def:pinv}
 The element $u_t$ defines the same class in $K_1(C^\ast(M)^\Gamma)$ for all $t\in [0, \infty)$.
We define the index class $\ind(D_{M}) = [u_t]\in K_1(C^\ast(M)^\Gamma)$.
\end{definition}

To summarize, for each $t\in [0, \infty)$, we have constructed a representative $u_t\in  (C^\ast( M)^\Gamma)^+$ of $\ind(D_{M})$ satisfying the following.  
\begin{enumerate}[(i)]
\item Recall that $\sn_t = \frac{1}{n}$ for $t\in [n, n+1)$. We have the propagation of $u_t$ restricted to $M_{(-\sn_t)} = M \backslash ((-\sn_t, 0]\times \partial M)$ is bounded by $2N\sn_t$, hence goes to $0$, as $t\to \infty$. Here we emphasize that the integer $N$ is fixed.  
\item Given $f\in L^2(M, S)$, if the support of $f$ is contained in a $\sn_t$-neighborhood of $\partial M$, then the support of $u_t(f)$ is contained in a $(2N\sn_t)$-neighborhood of $\partial M$. 
\end{enumerate} 
Informally speaking, as $t\to \infty$, the propagation of $u_t$ goes $0$ when away from a small neighborhood $\mathcal W_t$ of the boundary $\partial M$, while $\mathcal W_t$ is shrinking to the boundary. Moreover, the propagation of $u_t$ along  \textit{the normal direction} also goes to $0$ even in this neighborhood $\mathcal W_t$. However, note that in general we have no control over the propagation of $u_t$ along the $\partial M$-direction in $\mathcal W_t$.

\subsection{Idempotents} \label{sec:idem}
The even dimensional case is parallel to the odd dimensional case above. We 
will briefly go through the construction but leave out the details. 

Now we assume that $M_\infty = M\cup_{\partial M}( \mathbb R_{\geq 0}\times \partial M) $ is even dimensional. Then the associated Dirac operator $D_{M_\infty}$ is an odd operator with respect to the natural $\mathbb Z/2\mathbb Z$-grading on the spinor bundle $S$. Let $F_t $ be as in formula $\eqref{eq:fp}$ above. This time, $F_t$ is also odd-graded and let us write 
\[ F_t = \begin{pmatrix}  0 & V_t \\ U_t & 0\end{pmatrix} \]
with respect to the $\mathbb Z/2\mathbb Z$-grading. Notice that 
\[  F_t^2 = \begin{pmatrix} V_tU_t & 0 \\ 0 & U_tV_t\end{pmatrix}. \]
Now with respect to a decomposition $L^2(M_{(c)}, S) \oplus L^2(\mathbb R_{\geq c}\times \partial M, S)$ of the Hilbert space $ L^2(M_\infty, S)$, let us write 
\[ U_tV_t= \begin{pmatrix} B_{11}(t) & B_{12}(t) \\ B_{21}(t)  & B_{22}(t) \end{pmatrix} \quad \textup{and} \quad  V_tU_t= \begin{pmatrix} C_{11}(t) & C_{12}(t) \\ C_{21}(t)  & C_{22}(t) \end{pmatrix}.\]
We can choose $c$ sufficiently large so that the operator norms of $B_{12}(t), B_{21}(t),$ $ (B_{22}(t)- 1), C_{12}(t), C_{21}(t)$ and $ (C_{22}(t)-1)$ are all sufficiently small. Let us define 
\[ W_t = \begin{pmatrix} 1 & U_t \\ 0 & 1\end{pmatrix} \begin{pmatrix} 1 & 0 \\ -V_t & 1\end{pmatrix} \begin{pmatrix} 1 & U_t  \\ 0 & 1 \end{pmatrix}\begin{pmatrix} 0 & -1\\ 1 & 0 \end{pmatrix}\]
and
 \begin{align*}
 p_t &= W_t \begin{pmatrix} 1 & 0 \\ 0 & 0\end{pmatrix} W_t^{-1}\\
& = \begin{pmatrix} U_tV_t + U_tV_t(1-U_tV_t) & (2 + U_tV_t)U_t(1-V_tU_t)  \\ V_t(1-U_tV_t) & (1-U_tV_t)^2\end{pmatrix}.
\end{align*}
Now if we set 
\[ Z_1(t)= \begin{pmatrix} 1 - B_{11}(t)& 0 \\ 0  & 0 \end{pmatrix} \quad \textup{and} \quad  Z_2(t)= \begin{pmatrix} 1-C_{11}(t) & 0\\ 0  & 0 \end{pmatrix},\]
then we can assume that  $\| Z_1(t) - (1 - U_tV_t)\|$ and $\|Z_2(t) - (1- V_tU_t)\|$ are sufficiently small. Define
\[  q_t= \begin{pmatrix} 1 - Z_1(t)^2 & (2 + U_tV_t)U_tZ_2(t) \\ V_tZ_1(t) & Z_1(t)^2\end{pmatrix}. \]
Similar to the odd dimensional case, it is not difficult to see that there exists  $\Lambda > 0$ such that   $q_t$ preserves the decomposition  $L^2(M_{(\Lambda)}, S)\oplus L^2(\mathbb R_{\geq \Lambda}\times \partial M, S)$ 
and  $q_t = \left(\begin{smallmatrix} 1  & 0 \\ 0 & 0 \end{smallmatrix}\right) $ on  $L^2(\mathbb R_{\geq \Lambda} \times \partial M, S)$. So by restricting $q_t$ to $L^2(M_{(\Lambda)}, S)$, we see that $q_t$ naturally gives rise to an element in $ (C^\ast( M_{(\Lambda)})^\Gamma)^+$, which we will still denote by $q_t$. Moreover, if no confusion arises, we also denote the image of $q_t$ under the homomorphism $\Phi: (C^\ast( M_{(\Lambda)})^\Gamma)^+ \to (C^\ast( M)^\Gamma)^+$ by $q_t \in (C^\ast( M)^\Gamma)^+$. 

Notice that $q_t$ may not be an idempotent, but a quasi-idempotent in general (cf. \cite[Section 4]{GY98}). An element $q$ in a $C^\ast$-algebra $\mathcal A$ is called $\delta$-quasi-idempotent, if 
\[ \| q - q^2\| < \delta. \]
For sufficiently small $\delta$, say $\delta< 1/4$, a $\delta$-quasi-idempotent produces an idempotent by holomorphic functional calculus,  cf \cite[Section 2.2]{MR3122162}. 
  So $q_t$ defines a class in $K_0((C^\ast(M)^\Gamma)^+)$ for each $t\in [0, \infty)$.

\begin{definition} For all $t\in [0, \infty)$, 
 \[ [q_t] - [\begin{psmallmatrix} 1 & 0 \\ 0 & 0
 \end{psmallmatrix}]\] defines the same index class 
$\ind(D_{M}) \in K_0(C^\ast(M)^\Gamma)$.
\end{definition}

To summarize, for each $t\in [0, \infty)$, we have constructed a representative $q_t \in  (C^\ast( M)^\Gamma)^+$ of $\ind(D_{M})$ satisfying the following.  
\begin{enumerate}[(i)]
\item Recall that $\sn_t = \frac{1}{n}$ for $t\in [n, n+1)$. We have the propagation of $q_t$ restricted to $M_{(-\sn_t)} = M \backslash ((-\sn_t, 0]\times \partial M)$ is bounded by $5\sn_t$, hence goes to $0$, as $t\to \infty$.
\item Given $f\in L^2(M, S)$, if the support of $f$ is contained in a $\sn_t$-neighborhood of $\partial M$, then the support of $q_t(f)$ is contained in a $(5\sn_t)$-neighborhood of $\partial M$. Here we emphasize that the integer $N$ is fixed.  
\end{enumerate}

\begin{remark}
The same construction works if we switch to the maximal versions of all the $C^\ast$-algebras above.
\end{remark}

\section{Main theorem}\label{sec:main}

In this section, we prove the main theorem of this paper. 

Let $M$ be a complete spin manifold with boundary $\partial M$ such that
\begin{enumerate}[(i)]
\item the metric on $M$ has product structure near $\partial M$ and its restriction on $\partial M$, denoted by $h$, has positive scalar curvature;
\item there is a proper and cocompact isometric action of a discrete group $\Gamma$ on $M$;
\item the action of $\Gamma$ preserves the spin structure of $M$. 
\end{enumerate}  

Recall that we denote by $M_\infty = M\cup_{\partial M} (\mathbb R_{\geq 0} \times \partial M)$ and $M_{(n)} = M\cup_{\partial M}([0, n]\times \partial M)$ for $n\geq 0$  (cf. Section $\ref{sec:mbpsc}$).  We denote  the associated Dirac operator on $M_\infty$ by $D_{M_\infty}$ and the associated Dirac operator on $\partial M$ by $D_{\partial M}$. Let $m = \dim M$. 
Recall the following facts:
\begin{enumerate}[(a)]
\item by our discussion in Section $\ref{sec:mbpsc}$,  the operator $D_{M_\infty}$ actually defines an index class $\ind(D_{M})\in K_m(C^\ast(M)^\Gamma)$;
\item  by the discussion in Section $\ref{sec:rho}$, there is a higher rho invariant $\rho(D_{\partial M}, h)\in K_{m-1}(C_{L,0}(\partial M)^\Gamma)$ naturally associated to  $D_{\partial M}$;
\item the short exact sequence 
\[  0\to C^\ast_{L,0}(M)^\Gamma \to C^\ast_L(M)^\Gamma \to C^\ast(M)^\Gamma \to 0\]
induces the following long exact sequence
\[ \cdots \to K_i(C_L^\ast(M)^\Gamma) \to  K_i(C^\ast(M)^\Gamma) \xrightarrow{\partial_i } K_{i-1}(C^\ast_{L,0}(M)^\Gamma) \to  K_{i-1}(C^\ast_L(M)^\Gamma) \to \cdots; \]
\item the inclusion map $\iota: \partial M \to M$  induces a natural homomorphism (cf. Section $\ref{sec:roelocal}$)
\[ \iota_\ast: K_i(C^\ast_{L, 0}(\partial M)^\Gamma) \to K_i(C^\ast_{L,0}(M)^\Gamma);\] 
\item there are natural isomorphisms (cf. \cite{MR1451759}) 
\[ \xymatrix{&   &   & K_i(C^\ast_{L, 0} ( \partial M; M)^\Gamma )& \\
&  & K_i(C^\ast_{L, 0} (\partial M )^\Gamma) \ar[r]^-{\cong}  \ar[ru]^-{\cong}& K_{i}(C_{L,0}^\ast(\partial M; \mathbb R_{\leq 0}\times \partial M)^\Gamma) \ar[u]_-{\cong} & }  \]
where $\partial M$ embeds into $\mathbb R_{\leq 0}\times \partial M$ as the subset  $\{0\}\times \partial M $.
\end{enumerate}

We have the following main theorem of the paper. This extends a theorem of Piazza and Schick \cite[Theorem 1.17]{Piazza:2012fk} to all dimensions. We also point out that the same proof works equally for the real case (see Remark \ref{remark:dim} below). 

\begin{theorem}\label{thm:main}
We have that
\[ \partial_m(\ind(D_{M}) )  = \iota_\ast( \rho(D_{\partial M}, h) ) \]
in $K_{m-1}(C^\ast_{L, 0} (M )^\Gamma)$, where $\partial_m$ is the homomorphism $ K_m(C^\ast(M)^\Gamma) \to K_{m-1}(C^\ast_{L,0}(M)^\Gamma)$ and $\iota_\ast$ is the homomorphism $ K_i(C^\ast_{L, 0}(\partial M)^\Gamma) \to K_i(C^\ast_{L,0}(M)^\Gamma)$ as above.  
\end{theorem}
\begin{proof}
We will carry out a detailed proof for the odd dimensional case (i.e. when $m$ is odd). The proof for the even dimensional case is essentially the same.

By the discussion in Section $\ref{sec:mbpsc}$, we have a uniformly continuous path of invertible elements $u_t \in (C^\ast( M)^\Gamma)^+$ with the following properties. 
\begin{enumerate}[(i)]
\item $[u_t] = \ind (D_{M})$ for each fixed $t\in[0, \infty)$
\item Recall that $\sn_t = \frac{1}{n}$ for $t\in [n, n+1)$. We have the propagation of $u_t$ restricted to $M_{(-\sn_t)} = M \backslash ((-\sn_t, 0]\times \partial M)$ is bounded by $2N\sn_t$, hence goes to $0$, as $t\to \infty$. Here we emphasize that the integer $N$ is fixed.  
\item Given $f\in L^2(M, S)$, if the support of $f$ is contained in a $\sn_t$-neighborhood of $\partial M$, then the support of $u_t(f)$ is contained in a $(2N\sn_t)$-neighborhood of $\partial M$. 
\end{enumerate} 
Now let us choose a sufficiently large $t_0>0$ so that $2N\sn_{t_0} < \frac{1}{100}$. We will work with the path $\widetilde u_t = u_{t+t_0}$ instead of $u_t$. Let us write $\widetilde \sn_t = \sn_{t+t_0} $. Then the propagation of $\widetilde u_t$ restricted to $M_{(-\widetilde \sn_t)}$ is bounded by $2N\widetilde \sn_t$.

 We have $\varepsilon_t \leq \frac{1}{100}$ for all $t\in [0, \infty)$ and $\varepsilon_t$ goes to $0$, as $t\to \infty$. Now for each $t\in [0, \infty)$, choose a $\Gamma$-invariant partition of unity $\{\varphi_t, \psi_t\}$  on $M$ such that 
\begin{enumerate}[(a)]
\item $\varphi_t(x) + \psi_t(x) = 1$;  
\item $\varphi_t(x) \equiv 1$ on $M_{(-2\widetilde \sn_t)}$ and $\varphi_t(x) \equiv 0$ on $[-\widetilde \sn_t, 0]\times \partial M$;
\item $\psi_t \to 0$ pointwise, as $t\to \infty$.
\end{enumerate}
Now let $\theta\in C^\infty(-\infty, \infty)$ be a decreasing function such that $\theta|_{(-\infty, 1]}\equiv 1$ and $\theta|_{[2, \infty)} \equiv 0$.  
We define
\begin{equation}\label{eq:inv}
  U_t= \begin{cases} (1-t)\widetilde u_0 + t\left( \varphi_0^{1/2} \widetilde u_0 \varphi_0^{1/2} + \psi_0^{1/2} \widetilde u_0 \psi_0^{1/2}\right) & \textup{if $t\in [0, 1] $,} \\
1 + \varphi_{t-1}^{1/2} (\widetilde u_{t-1} -1) \varphi_{t-1}^{1/2} + \theta(t)\psi_{t-1}^{1/2} (\widetilde u_{t-1} -1) \psi_{t-1}^{1/2}  & \textup{if $t\in [1, \infty) $.}
\end{cases}
\end{equation}
Notice that $U_t \in (C^\ast(M)^\Gamma)^+$ and the propagation of $U_t$ (on the entire $M$) goes to $0$, as $t\to \infty$. In other words, the path $U_t$, $ 0\leq t <\infty$, is an element in $(C_L^\ast(M)^\Gamma)^+$. So  the path 	 $U_t, 0\leq t <\infty,$ is a lift of $\widetilde u_0\in (C^\ast( M)^\Gamma)^+$ for  the short exact sequence 
\[ 0 \to C_{L,0}^\ast(M)^\Gamma \to C_L^\ast( M)^\Gamma \to C^\ast( M)^\Gamma \to 0. \]
We also apply the same argument to $v_t = u_t^{-1}$ and define similarly $ V_t \in (C^\ast( M)^\Gamma)^+$ for $ 0\leq t < \infty$.
 Now we define an idempotent 
\begin{equation} \label{eq:idem}
 \scalebox{0.9}{ $p_t= \begin{pmatrix} U_tV_t + U_tV_t(1-U_tV_t) & (2 + U_tV_t) U_t (1-V_tU_t)  \\ V_t(1-U_tV_t) & (1-U_tV_t)^2\end{pmatrix}.$ }
\end{equation}
By definition, we have
\begin{equation}\label{eq:ind}
\partial_1 (\ind(D_{M}) ) = [p_t] - \left[\left(\begin{smallmatrix} 1 & 0 \\0 & 0\end{smallmatrix}\right) \right] \in K_0(C_{L, 0}^\ast(M)^\Gamma).
\end{equation}
\begin{lemma}\label{lm:decomp}
For all $t\in [0, \infty)$, the element $p_t$ preserves the decomposition 
\[  L^2(M, S) = L^2(M_{(-30N\widetilde \sn_t)}, S)\oplus  L^2([-30N\widetilde \sn_t,0]\times \partial M, S)\]
and $p_t = \left(\begin{smallmatrix} 1 & 0 \\0 & 0\end{smallmatrix} \right)$ on $ L^2( M_{(-30N\widetilde \sn_t)}, S)$, where $M_{(-30N\widetilde \sn_t)}= M\backslash ((-30N\widetilde \sn_t, 0]\times \partial  M)$.
\end{lemma}
\begin{proof}
Notice that the propagations of $U_t$ and $V_t$ are bounded by $2N\widetilde \sn_t$. Therefore, if $f\in L^2(M_{(-30N\widetilde \sn_t)}, S)$, then the support of $p_t(f)$ is contained in $M_{(-20N\widetilde \sn_t)}$. However, $U_t$ and $V_t$ are genuinely the inverse of each other on $M_{(-20N\widetilde \sn_t)}$. It follows that 
\[ p_t = \begin{pmatrix} 1 & 0 \\0 & 0\end{pmatrix}\] on $ L^2( M_{(-30N\widetilde \sn_t)}, S)$.

Now we further decompose $ L^2([-30N\widetilde \sn_t,0]\times \partial M, S)$ into 
\[   L^2([-30N\widetilde \sn_t, -15N\widetilde \sn_t]\times \partial M, S) \oplus L^2([-10N\widetilde \sn_t,0]\times \partial M, S). \]
If $f_1 \in L^2([-30N\widetilde \sn_t, -15N\widetilde \sn_t]\times \partial M, S)$, then the support of $p_t(f_1)$ is contained in $M_{(-5N\widetilde \sn_t)}$. Again $U_t$ and $V_t$ are genuinely the inverse of each other on $M_{(-5N\widetilde \sn_t)}$. Therefore, we see that $ p_t(f_1) = f_1$, which implies that 
\[ \supp(p_t(f_1))\subseteq [-30N\widetilde \sn_t, -15N\widetilde \sn_t] \times \partial M \subseteq [-30N\widetilde \sn_t,0]\times \partial M. \] 
Now if $f_2\in L^2([-15N\widetilde \sn_t,0]\times \partial M, S)$, then clearly 
\[ \supp(p_t(f_2))\subseteq [-25N\widetilde \sn_t,0]\times \partial M \subseteq [-30N\widetilde \sn_t,0]\times \partial M.\] 
This finishes the proof of the lemma.
\end{proof}
 
We see that the restriction of $p_t$ to the component$ L^2( M_{(-30N\widetilde \sn_t)}, S)$ gives trivial $K$-theoretical information. Therefore, without loss of information, we can think of $p_t$ as an element in $(C^\ast_{L, 0} ([-30N\widetilde \sn_t,0]\times \partial M)^\Gamma)^+$. In fact, choose the space $[-2, 0]\times \partial M$, and  consider the restriction  $p_t$ to $ L^2( [-2, 0]\times \partial M, S)$, still denoted by $p_t$ if no confusion arises. By Lemma $\ref{lm:decomp}$, we see that the path $p_t$ in fact defines an element in   $(C^\ast_{L, 0} (\partial M; [-2,0]\times \partial M)^\Gamma)^+$, where $\partial M$ embeds into $[-2,0]\times \partial M$ as $\{0\}\times \partial M$. 

To summarize, we have in fact
\begin{equation}
[p_t] - \left[\left(\begin{smallmatrix} 1 & 0 \\0 & 0\end{smallmatrix}\right) \right] \in  K_0(C^\ast_{L, 0} (\partial M; [-2,0]\times \partial M)^\Gamma ).
\end{equation} 
Moreover, Equation $\eqref{eq:ind}$ can be made more precise as follows: 
\begin{equation}\label{eq:lidem}
\partial_1 (\ind(D_{M}) )  = \iota_\ast( [p_t] - \left[\left(\begin{smallmatrix} 1 & 0 \\0 & 0\end{smallmatrix}\right) \right]), 
\end{equation}
 where $\iota_\ast$ is the natural homomorphism $\iota_\ast:   K_0(C^\ast_{L, 0} (\partial M; [-2,0]\times \partial M)^\Gamma )\to  K_0(C^\ast_{L, 0} (M)^\Gamma )$.

Now let us turn to the Dirac operator $D_{\mathbb R\times \partial M}$ on $\mathbb R\times \partial M$ for the moment. First let us fix some notation. We define 
\[ \mathcal A_+ = C_{L, 0}^\ast(\mathbb R_{\geq 0} \times \partial M)^\Gamma, \quad \mathcal A_- = C_{L, 0}^\ast(\mathbb R_{\leq 0}\times \partial M)^\Gamma, \]
and
\[ \mathcal J_- =  C_{L, 0}^\ast(\partial M ; \mathbb R_{\leq 0}\times {\partial M})^\Gamma, \]
where $\partial M = \{0\}\times \partial M$ as a subset of $\mathbb R_{\leq 0} \times \partial M$.

Notice that the following $C^\ast$-algebras are naturally isomorphic to each other:
\begin{align*} 
\mathcal A_-/\mathcal J_- &\cong C_{L, 0}^\ast(\mathbb R\times \partial M)^\Gamma /  C_{L, 0}^\ast(\mathbb R_{\geq 0}\times\partial M ; \mathbb R\times \partial M)^\Gamma\\
& \cong C_{L, 0}^\ast(\mathbb R_{\leq 0}\times \partial M; \mathbb R\times \partial M)^\Gamma /  C_{L, 0}^\ast(\partial M ; \mathbb R\times \partial M)^\Gamma.
\end{align*}
Consider the following short exact sequence of $C^\ast$-algebras
\begin{equation*}
 0 \to C_{L,0}^\ast( \partial M; \mathbb R\times \partial M)^\Gamma\to  {\substack{C_{L,0}^\ast(\mathbb R_{\geq 0}\times \partial M; \mathbb R\times \partial M)^\Gamma \\ \oplus  \\ C^\ast_{L,0}(\mathbb R_{\leq 0}\times \partial M; \mathbb R\times \partial M)^\Gamma}} \to C_{L,0}^\ast(\mathbb R\times \partial M)^\Gamma \to 0,
\end{equation*}
where $\partial M = \{0\}\times \partial M$ as a subset of $\mathbb R\times \partial M$.  It is clear that the following diagram commutes:
\begin{equation*}
\resizebox{1\hsize}{!}{$\xymatrix{
0 \ar[r] & C_{L,0}^\ast( \partial M; \mathbb R\times \partial M)^\Gamma\ar[r] \ar@{=}[d] & {\substack{C_{L,0}^\ast(\mathbb R_{\geq 0}\times \partial M; \mathbb R\times \partial M)^\Gamma \\ \oplus \\ C^\ast_{L,0}(\mathbb R_{\leq 0}\times \partial M; \mathbb R\times \partial M)^\Gamma } } \ar[r]\ar[d] & C_{L,0}^\ast(\mathbb R\times \partial M)^\Gamma \ar[r]\ar[d] &  0\\
0 \ar[r] & C_{L,0}^\ast( \partial M; \mathbb R\times \partial M)^\Gamma  \ar[r]  & C^\ast_{L,0}(\mathbb R_{\leq 0}\times \partial M; \mathbb R\times \partial M)^\Gamma \ar[r]  & \mathcal A_-/\mathcal J_- \ar[r] &  0 \\
0\ar[r]  & \mathcal J_- \ar[r] \ar[u] & \mathcal A_- \ar[u] \ar[r] & \mathcal A_-/\mathcal J_- \ar[r]\ar@{=}[u] & 0  
}$ }
\end{equation*}
where the maps are all defined the obvious way. Recall that (cf. \cite{MR1451759})
\[ K_i(\mathcal A_+) = K_i( C_{L,0}^\ast(\mathbb R_{\geq 0}\times \partial M; \mathbb R\times \partial M)^\Gamma),\]
\[ K_i(\mathcal A_-) = K_i( C_{L,0}^\ast(\mathbb R_{\leq 0}\times \partial M; \mathbb R\times \partial M)^\Gamma),\]
and 
\[ K_i(\mathcal J_-) = K_i(C_{L,0}^\ast(\partial M)^\Gamma) = K_i( C_{L,0}^\ast( \partial M; \mathbb R\times \partial M)^\Gamma).  \]
It follows that the second row and the third row of the above commutative diagram give rise to identical long exact sequences at $K$-theory level.

Therefore, we have the following commutative diagram. 
\begin{equation*}
\resizebox{1\hsize}{!}{$
 \xymatrix{  
 {\substack{C_{L,0}^\ast(\mathbb R_{\geq 0}\times \partial M; \mathbb R\times \partial M)^\Gamma \\ \oplus \\ C^\ast_{L,0}(\mathbb R_{\leq 0}\times \partial M; \mathbb R\times \partial M)^\Gamma } }\ar[r]  \ar[dd]^-{\Theta_1} & K_1 (C_{L, 0}^\ast(\mathbb R\times \partial M)^\Gamma)\ar[r]^-{\partial_{MV}} \ar[dd]^-{\Theta_2} & K_0(C_{L,0}^\ast( \partial M; \mathbb R\times \partial M)^\Gamma ) \ar[dd]^-{\Theta_3} _-{\cong}    & K_0(C_{L,0}^\ast(\partial M)^\Gamma) \ar[l]_-{\cong} \ar[d]^{\cong}  \\
 &  &  &   K_0(C_{L,0}^\ast(\partial M; [-2, 0]\times \partial M)^\Gamma) \ar[ul]_{\cong} \ar[dl]_{\cong} \ar[d]^{\iota_\ast}\\
K_1 (\mathcal A_-) \ar[r]&  K_1 (\mathcal A_-/\mathcal J_-) \ar[r]^-{\partial_{A/J} } &  K_0(\mathcal J_-) \ar@{-->}[r]_{\sigma_\ast} & K_0(C_{L,0}^\ast(M)^\Gamma) \\
}$
 }
\end{equation*}
Here the morphism $\sigma_\ast$ needs an explanation. It is defined to be
\[ \sigma_\ast : K_0(\mathcal J_-) \xrightarrow{\cong} K_0(C_{L,0}^\ast(\partial M; [-2, 0]\times \partial M)^\Gamma) \xrightarrow{\iota_\ast}  K_0(C_{L,0}^\ast(M)^\Gamma).\]
  Notice that  the morphism $\sigma_\ast$ is only well-defined at the level of $K$-theory, as there is no natural homomorphism from $\mathcal J_-$ to $C_{L,0}^\ast(M)^\Gamma$.

Let $dx^2+h$ be the standard product metric on $\mathbb R\times \partial M$ induced by the metric $h$ on $\partial M$. We have the following:
\begin{enumerate}[(1)]
\item the higher rho invariant $\rho(D_{\mathbb R \times \partial M}, dx^2 + h)$ lies in $K_1 (C_{L, 0}^\ast(\mathbb R\times \partial M)^\Gamma)$;
\item $\rho(D_{\partial M}, h)$ lies in $K_0(C_{L,0}^\ast(\partial M)^\Gamma)$;
\item $\partial_0 (\ind(D_{M}))  $ lies in $K_0(C_{L,0}^\ast(M)^\Gamma)$; more precisely (cf. Formula $\eqref{eq:lidem}$), 
 \[  \partial_1 (\ind(D_{M}))  = \iota_\ast( [p_t] - \left[\left(\begin{smallmatrix} 1 & 0 \\0 & 0\end{smallmatrix}\right) \right]).  \]
\end{enumerate}

Under each isomorphism, labeled by $\cong$, in the commutative diagram above,  the element $\rho(D_{\partial M}, h)$ in $K_0(C_{L,0}^\ast(\partial M)$ gets mapped naturally to an element in the corresponding $K$-theory group. For notational simplicity, let us denote all these elements by $\rho(D_{\partial M}, h)$, if no confusion arises.

Now the theorem is reduced to the following claim. 
\begin{claim*}We have that 
\begin{enumerate}[(i)]
\item $\sigma_\ast(\partial_{A/J}\circ \Theta_2 \left[\rho(D_{\mathbb R \times \partial M}, dx^2 + h)\right] ) =  \partial_1(\ind(D_M)) $ in $K_0(C_{L,0}^\ast(M)^\Gamma)$;
\item  $ \partial_{MV}\left[ \rho(D_{\mathbb R \times \partial M}, dx^2 + h)\right]  = \rho (D_{\partial M}, h)$ in $K_0(C_{L,0}^\ast(\partial M; \mathbb R\times \partial M)^\Gamma)$. 
\end{enumerate}
\end{claim*}
Indeed, assuming the claim for the moment, then the commutativity of the above diagram implies that 
\begin{align*}
\partial_1(\ind(D_M)) &=\sigma_\ast(\partial_{A/J}\circ \Theta_2 \left[\rho(D_{\mathbb R \times \partial M}, dx^2 + h)\right] )\\
 & = \sigma_\ast( \Theta_3\circ \partial_{MV} \left[\rho(D_{\mathbb R \times \partial M}, dx^2 + h)\right] ) \\
& = \sigma_\ast(\rho(D_{\partial M}, h) ) \\
& = \iota_\ast(\rho(D_{\partial M}, h) ).
\end{align*}
So now let us prove the above claim.
\begin{enumerate}[(i)]

\item First we will prove that 
\[ \sigma_\ast(\partial_{A/J}\circ \Theta_2 \left[\rho(D_{\mathbb R \times \partial M}, dx^2 + h)\right] ) =  \partial_1(\ind(D_M))\]
in $K_0(C_{L,0}^\ast(M)^\Gamma)$. In fact, we will find an explicit representative $q_t$, $0\leq t <\infty$, of the class 
\[ \partial_{A/J}\circ \Theta_2 \left[\rho(D_{\mathbb R \times \partial M}, dx^2 + h)\right] \in K_0(\mathcal J_-)\] such that  the path $q_t$ lies in $(C_{L,0}^\ast(\partial M; [-2, 0]\times \partial M)^\Gamma)^+$ (modulo trivial components) and $q_t = p_t$, where $p_t$ is the idempotent in Formula $\eqref{eq:idem}$. The method of constructing $q_t$ is essentially to repeat the construction of $p_t$. While $p_t$ is obtained from operators on $L^2(M_\infty, S)$, we would work with operators on $L^2(\mathbb R\times \partial M, S)$ this time. We first repeat the construction of $u_t$ as in Section $\ref{sec:invertible}$ for\footnote{$F$ is well-defined, since the metric $dx^2+ h$ is uniformly positive on $\mathbb R\times \partial M$.} $F = \frac{D_{\mathbb R\times \partial M}}{\sqrt{D_{\mathbb R\times \partial M}^2}}$. Informally speaking, this allows us to chop off $\mathbb R_{\geq 0}\times \partial M$. In particular, if we denote the resulting path of invertible elements by $w_t$, then $w_t \in C^\ast(\mathbb R_{\leq 0}\times \partial M)$ for each $t\in [0,\infty)$. Moreover, $w_t$ satisfies similar  properties that $u_t$ has (see the comments following Definition $\ref{def:pinv}$). The reader may have noticed that in general the path $w_t$ does not give rise to an element in $\mathcal A_- = C_{L,0}^\ast(\mathbb R_{\leq 0}\times \partial M)^\Gamma$, due to ``bad'' propagation control near the boundary $\partial M = \{0\}\times \partial M$. However, notice that the ``bad" part of $w_t$ only lives near the boundary $\partial M$. We can artificially get rid of this ``bad'' part by for example using an appropriate cut-off function (also compare this to the construction of $U_t$ in Formula $\eqref{eq:inv}$).   Of course, different choices of cut-off functions will produce different elements in general. But this choice become irrelevant once we pass to the quotient $\mathcal A_-/\mathcal J_-$, where $\mathcal J_- =C_{L,0}^\ast(\partial M; \mathbb R_{\leq 0}\times \partial M)^\Gamma$.  So now we repeat the construction of $p_t$ but using $w_t$ instead of $u_t$ this time. Again informally speaking, this allows us to chop off $\mathbb R_{\leq -2}\times \partial M$, where in the case of $u_t$ we chopped off $M_{(-2)}$. We define $q_t$ to be the resulting idempotent for each $t\in [0,\infty)$. Note that the Dirac operator $D_{\mathbb R\times \partial M}$ clearly coincides with the Dirac operator $D_{M_\infty}$ on $(-3, \infty)\times \partial M$. Because of finite propagation property, it follows that  $q_t = p_t$ on $[-2,0]\times \partial M$, for all $t\in [0,\infty)$. This completes the proof of this part. 

\item Now let us prove that \[ \partial_{MV}\left[ \rho(D_{\mathbb R \times \partial M}, dx^2 + h)\right]  = \rho (D_{\partial M}, h).\] 
We shall use the Kasparov $KK$-theory formulation. We refer the reader to Section $\ref{sec:kk}$ above for more details. Recall that there are natural isomorphisms (cf. formula $\eqref{eq:k-kk}$)
\[ \vartheta_1 :  K_1(C^\ast_L(\mathbb R)) \cong KK^1 (\mathbb C ,  C_L^\ast(\mathbb R)),\]
\[\vartheta_2 :   K_0(C_{L, 0}^\ast( \partial M)^\Gamma) \cong KK(\mathbb C , C_{L, 0}^\ast( \partial M)^\Gamma),\]
and
\[ \vartheta_3 :  K_1(C^\ast_L(\mathbb R) \otimes C_{L, 0}^\ast( \partial M)^\Gamma) \cong KK^1(\mathbb C , C^\ast_L(\mathbb R) \otimes C_{L, 0}^\ast( \partial M)^\Gamma).\]
Moreover, the following diagram commutes
\[ 
 \xymatrix{  K_1(C^\ast_L(\mathbb R)) \times K_0(C_{L,0}^\ast(\partial M)^\Gamma)  \ar[r]^{\otimes_K} \ar[d]_{\cong}^{\vartheta_1\times \vartheta_2} & K_1(C^\ast_L(\mathbb R)\otimes C_{L,0}^\ast(\partial M)^\Gamma) \ar[d]_{\cong}^{\vartheta_3}  \\
 KK^1 (\mathbb C, C_L^\ast(\mathbb R)) \times KK(\mathbb C, C_{L, 0}^\ast( \partial M)^\Gamma)) \ar[r]^-{\otimes_{KK}} & KK^1(\mathbb C, C_L^\ast(\mathbb R) \otimes C_{L,0}^\ast( {\partial M})^\Gamma ) 
} \]
where $\otimes_K$ is the standard external product in $K$-theory and $\otimes_{KK}$ is the Kasparov product in $KK$-theory. In other words, the commutative diagram states that the Kasparov product is compatible with the external $K$-theory product.

Now consider 
\[  v_0= \ind_L(D_{\mathbb R}) \in K_1(C^\ast_L(\mathbb R)), \quad  \rho_0 = \rho(D_{\partial M}, h) \in K_0(C_{L,0}^\ast(\partial M)^\Gamma), \]
and 
\[  \rho_1 = \rho(D_{\mathbb R\times \partial M}, dx^2 + h) \in K^1( C_{L, 0}^\ast(\mathbb R\times {\partial M})^\Gamma ),   \]
where $D_{\mathbb R}$ is the Dirac operator on $\mathbb R$. We denote by $\iota$ the natural homomorphism 
\[ \iota:  C_L^\ast(\mathbb R) \otimes C_{L,0}^\ast( {\partial M})^\Gamma \to  C_{L, 0}^\ast(\mathbb R\times {\partial M})^\Gamma.\]
By Claim $\ref{claim:rho}$ in Section $\ref{sec:kk}$, we have 
\[ \iota_\ast [ v_0 \otimes_K   \rho_0]  = [\rho_1]   \] 
in $ K_1(C^\ast_{L, 0}(\mathbb R \times \partial M)^\Gamma)$. So it remains to show that 
\[ \partial_{MV} (\iota_\ast [ v_0 \otimes_K   \rho_0]) = \rho_0. \] 
By the following commutative diagram
\[
 \scalebox{0.85}{\xymatrix{
0 \ar[r] & C_{L}^\ast(\{0\}; \mathbb R) \otimes C_{L,0}^\ast(\partial M)^\Gamma\ar[r] \ar[d]^\iota & {\substack{ C_{L}^\ast(\mathbb R_{\geq 0}; \mathbb R) \otimes C_{L,0}^\ast(\partial M)^\Gamma \\ \oplus \\ C^\ast_{L}(\mathbb R_{\leq 0}; \mathbb R) \otimes C_{L,0}^\ast(\partial M)^\Gamma}}\ar[r]  \ar[d]^{\iota \oplus \iota}&  C_{L}^\ast(\mathbb R) \otimes C_{L,0}^\ast(\partial M)^\Gamma \ar[r] \ar[d]^{\iota} & 0 \\
0 \ar[r] & C_{L,0}^\ast( \{0\}\times\partial M; \mathbb R\times \partial M)^\Gamma\ar[r]  & {\substack{C_{L,0}^\ast(\mathbb R_{\geq 0}\times \partial M; \mathbb R\times \partial M)^\Gamma \\ \oplus \\ C^\ast_{L,0}(\mathbb R_{\leq 0}\times \partial M; \mathbb R\times \partial M)^\Gamma } } \ar[r] & C_{L,0}^\ast(\mathbb R\times \partial M)^\Gamma \ar[r]  &  0
} }
\]
it is equivalent to prove that $\partial_{MV}(v_0 \otimes_K   \rho_0) = \rho_0.$
Now recall that 
\[ v_0 \otimes_K \rho_0 = v_0 \otimes \rho_0 + 1\otimes (1-\rho_0) \in K_1(C_{L}^\ast(\mathbb R) \otimes C_{L,0}^\ast(\partial M)^\Gamma ). \] 
%Consider the following two Mayer-Vietoris sequences: 
%\[ 0 \to C_{L}^\ast(\{0\}) \otimes C_{L,0}^\ast(\partial M)^\Gamma\to {\substack{ C_{L}^\ast(\mathbb R_{\geq 0}) \otimes C_{L,0}^\ast(\partial M)^\Gamma \\ \oplus \\ C^\ast_{L}(\mathbb R_{\leq 0}) \otimes C_{L,0}^\ast(\partial M)^\Gamma}} \to C_{L}^\ast(\mathbb R) \otimes C_{L,0}^\ast(\partial M)^\Gamma \to 0 \]
%and 
Let us also write $\partial_{MV}$ for the Mayer-Vietoris boundary map in the long exact sequence induced by 
\[ 0 \to C_{L}^\ast(\{0\}; \mathbb R) \to C_{L}^\ast(\mathbb R_{\geq 0}; \mathbb R) \oplus C^\ast_{L}(\mathbb R_{\leq 0}; \mathbb R) \to C_{L}^\ast(\mathbb R) \to 0. \]
It will be clear from the context which one we are using.

Now we can compute $\partial_{MV} (v_0 \otimes \rho_0 + 1\otimes (1-\rho_0) )$ explicitly by formula $\eqref{eq:keven}$ in Section $\ref{sec:kt}$. Indeed, we can lift $v_0 \otimes \rho_0 + 1\otimes (1-\rho_0) $ and its inverse to 
\[
  U =  x \otimes \rho_0 + 1\otimes (1-\rho_0) \quad \textup{and} \quad
  V  =  y \otimes \rho_0 + 1\otimes (1-\rho_0),
\]
where $x$ (resp. $y$)  is a lift of $v_0$ (resp. the inverse of $v_0$) in $C_{L}^\ast(\mathbb R_{\geq 0}; \mathbb R) \oplus C^\ast_{L}(\mathbb R_{\leq 0}; \mathbb R)$.
Now apply formula $\eqref{eq:keven}$ to $U$ and $V$. By a straightforward calculation, we see on the nose that 
\begin{align*}
 P &= \begin{pmatrix} UV + UV(1\otimes 1 -UV) & (2 + UV)(1\otimes 1-UV) U \\ V(1\otimes 1-UV) & (1\otimes 1-UV)^2\end{pmatrix} \\
 & =  \partial_{MV} (v_0) \otimes \rho_0 + \begin{pmatrix} 1 & 0 \\ 0 & 0 \end{pmatrix} \otimes (1-\rho_0). 
\end{align*}
where $\partial_{MV} (v_0) = \left( \begin{smallmatrix} xy + xy(1-xy) & (2 + xy)(1-xy) x \\ y(1-xy) & (1-xy)^2\end{smallmatrix}\right)$. 
It follows that 
\[ [P] - \left[\left(\begin{smallmatrix} 1 & 0 \\ 0 & 0\end{smallmatrix}\right)\right] =  ([\partial_{MV}(v_0)] - \left[\left(\begin{smallmatrix} 1 & 0 \\ 0 & 0\end{smallmatrix}\right)\right])\otimes \rho_0. \]  Notice that $([\partial_{MV}(v_0)] - \left[\left(\begin{smallmatrix} 1 & 0 \\ 0 & 0\end{smallmatrix}\right)\right]) $ is the generator $1 \in  K_0(C_{L}^\ast(\{0\}; \mathbb R))  = \mathbb Z$ by Bott periodicity. It follows that 
\[\partial_{MV}(v_0 \otimes_K   \rho_0) = \rho_0.\]
This finishes the proof. 
\end{enumerate}

\end{proof}

\begin{remark}\label{remark:dim}
To deal with manifolds of dimension $m \neq 0 \pmod 8$ in the real case, we work with $\cl_m$-linear Dirac operators, cf. \cite[Section II.7]{BLMM89}. Here $\cl_{m}$ is the standard real Clifford algebra on $\mathbb R^m$ with $e_ie_j + e_je_i =  - 2 \delta_{ij}$. We recall the definition of $\cl_m$-linear Dirac operators in the following. Consider the standard representation $\ell$ of $\textup{Spin}_m$ on $\cl_m$ given by left multiplication. Let $M$ be the manifold from the beginning of this section and $P_{\textup{spin}}(M)$ be the principal $\textup{Spin}_m$-bundle, then we define $\mathfrak S$ to be the vector bundle
\[ \mathfrak S = P_{\textup{spin}}(M)\times_{\ell} \cl_m. \]
We  denote the associated $\cl_m$-linear Dirac operator on $M_\infty$ by 
\[ \mathfrak D: L^2(M_\infty, \mathfrak S) \to L^2(M_\infty, \mathfrak S).\]
Notice that the right multiplication of $\cl_m$ on $\mathfrak S$ commutes with $\ell$. Moreover, $\mathfrak D$ is invariant under the action of $\Gamma$. 
Then, by the same argument as in Section $\ref{sec:mbpsc}$,  $\mathfrak D$ defines a  higher index class
\[\ind(\mathfrak D)\in \widehat K_0(C^\ast( M,  \mathbb R)^\Gamma\widehat\otimes \cl_m) \cong \widehat K_{m}(C^\ast( M, \mathbb R)^\Gamma) \cong K_m(C^\ast( M, \mathbb R)^\Gamma),\] 
where $C^\ast(M, \mathbb R)^\Gamma$ stands for the $\Gamma$-invariant Roe algebra of $ M$ with coefficients in $\mathbb R$. Moreover,  $\widehat K_\ast$ stands for the $\mathbb Z_2$-graded $K$-theory of $C^\ast$-algebras, $\widehat\otimes$ stands for $\mathbb Z_2$-graded tensor product, cf. \cite[Chapter III]{MK78}. Notice that for a trivially graded $C^\ast$-algebra $\mathcal A$, we have 
\[ \widehat K_{m}(\mathcal A) \cong K_m(\mathcal A).\] 
In the case of a manifold without boundary, i.e. $\partial M = \emptyset$, then we can define the local index class of $\mathfrak D$ and denote it by
\[\ind_L(\mathfrak D)\in \widehat K_0(C^\ast_L(M,  \mathbb R)^\Gamma\widehat\otimes \cl_m) \cong \widehat K_{m}(C_L^\ast(M, \mathbb R)^\Gamma) \cong K_m(C_L^\ast( M, \mathbb R)^\Gamma),\]
whose image under the evaluation map $\ev_\ast: K_m(C_L^\ast(M, \mathbb R)^\Gamma) \to K_m( C^\ast(M, \mathbb R)^\Gamma)$ is $\ind(\mathfrak D)$. Moreover, the higher rho invariant is also defined similarly in the real case. 
%Moreover, if in addition the whole $M$ is endowed with a metric $g$ of positive scalar curvature, then there is a natural higher rho invariant $\rho (\mathfrak D, g)$ associated to $\mathfrak D$  in  
%\[ \widehat K_0(C^\ast_{L, 0}( M,  \mathbb R)^\Gamma\widehat\otimes \cl_m) \cong \widehat K_{m}(C_{L,0}^\ast(M, \mathbb R)^\Gamma) \cong K_m(C_{L,0}^\ast( M, \mathbb R)^\Gamma).\] 
 
\end{remark}

\begin{remark}
Theorem $\ref{thm:main}$ also holds if we use the maximal versions of all the $C^\ast$-algebras in the theorem. 
\end{remark}

The following corollaries are immediate consequences of Theorem $\ref{thm:main}$.  In a similar context, these have already appeared in the work of Lott \cite{JLott92}, Botvinnik and Gilkey \cite{MR1339924}, and Leichtnam and Piazza \cite{ELPP01}.

\begin{corollary}\label{cor:ext}
With the same notation as above, if $\rho(D_{\partial M}, h)\neq 0$, then there does not exist a $\Gamma$-invariant complete Riemannian metric $g$ on $M$ with product structure near the boundary $\partial M$ such that $g$ has positive scalar curvature and $g|_{\partial M} =  h$.  
\end{corollary}

In other words, nonvanishing of the higher rho invariant is an obstruction to extension of the positive scalar curvature metric from the boundary to the whole manifold. 

Let $N$ be a spin manifold without boundary, equipped with a proper cocompact action of a discrete group $\Gamma$. Denote by $\mathcal R^+(N)^\Gamma$ the space of all $\Gamma$-invariant complete Riemannian metrics of positive scalar curvature on $ N$. For $h_0, h_1 \in \mathcal R^+(N)^\Gamma$, we say  $h_0$ and $h_1$ are path connected in $\mathcal R^+(N)^\Gamma$ if $h_0$ and $h_1$ are connected by a smooth path of $h_t \in \mathcal R^+(N)^\Gamma$, $0\leq t \leq t$. More generally, we say $h_0$ and $h_1$ are $\Gamma$-bordant if there exists a spin manifold $W$ with a proper cocompact action of $\Gamma$ such that $\partial W = N \amalg (-N)$ and $W$ carries a $\Gamma$-invariant complete Riemannian metric $g$ of positive scalar curvature\footnote{We assume the metric $g$ has product structure near the boundary.} with $g|_N = h_0$ and $g|_{-N} = h_1$. Here $-N$ means $N$ with the opposite orientation. Clearly, if $h_0$ and $h_1$ are path connected, then they are $\Gamma$-bordant.  

\begin{corollary}\label{cor:diff}
Let  $h_0, h_1\in R^+(N)^\Gamma$ as above.  If $\rho(D_{N}, h_0) \neq \rho(D_{N}, h_1)$, then $h_0$ and $h_1$ are not $\Gamma$-bordant. In particular, this implies that $h_0$ and $h_1$ are in different connected components of $\mathcal R^+(N)^\Gamma$.
\end{corollary}

\section{Stolz' positive scalar curvature exact sequence}\label{sec:stolz}
In this section, we apply our main theorem (Theorem $\ref{thm:main}$ above) to map the Stolz' positive scalar curvature exact sequence \cite{SS95} to a long exact sequence of $K$-theory of $C^\ast$-algebras. 

First, let us recall the Stolz' positive scalar curvature exact sequence. 

\begin{definition}
Given a topological space $X$, we denote by $\Omega^{\spin}_n(X)$ the set of bordism classes of pairs $(M, f)$, where $M$ is an $n$-dimensional closed spin manifold and $f: M\to X$ is a continuous map. Two such pairs $(M_1, f_1)$ and $(M_2, f_2)$ are bordant if there is a bordism $W$ between $M_1$ and $M_2$ (with compatible spin structure), and a continuous map $F: W\to X$ such that $F|_{M_i} = f_i$. Then $\Omega_n^\spin(X)$ is an abelian group with the addition being disjoint union. We call $\Omega_n^\spin(X)$ the $n$-dimensional spin bordism of $X$.  
\end{definition}

\begin{definition}
Let $\pos_n(X)$ be the bordism group of triples $(M, f, g)$, where $M$ is an $n$-dimensional closed spin manifold, $f: M\to X$ is a continuous map, and $g$ is a positive scalar curvature metric on $M$. Two such triples $(M_1, f_1, g_1)$ and $(M_2, f_2, g_2)$ are bordant if there is a bordism $(W, F, G)$ such that $G$ is a positive scalar curvature metric on $W$ with product structure near $M_i$ and  $G|_{M_i} = g_i$, and $F|_{M_i} = f_i$. 
\end{definition}

Then it is clear that forgetting the metric gives a homomorphism 
\[\pos_n(X) \to \Omega_n^\spin (X).\]

\begin{definition}
$R^\spin_n(X)$ is the bordism group of the triples $(M, f, h)$, where $M$ is an $n$-dimensional spin manifold (possibly with boundary), $f: M\to X$ is a continuous map, and $h$ is a positive scalar curvature metric on the boundary $\partial M$. Two triples $(M_1, f_1, h_1)$ and $(M_2, f_2, h_2)$ are bordant if 
\begin{enumerate}[(a)]
\item there is a bordism $(V, F, H)$ between $(\partial M_1, f_1, h_1)$ and $(\partial M_2, f_2, h_2)$ (considered as triples in $\pos_{n-1}(X)$),
\item and the closed spin manifold $M_1\cup_{\partial M_1} V \cup_{\partial M_2} M_2$ (obtained by gluing $M_1, V$ and $M_2$ along their common boundary components) is the boundary of a spin manifold $W$ with a map $E: W\to X$ such that $E|_{ M_i} = f_i$ and $E|_{V} = F$. 
\end{enumerate} 
\end{definition}
It is not difficult to see that the three groups defined above fit into the following long exact sequence: 
\[ 
\xymatrix{  \ar[r] &\Omega^\spin_{n+1}(X) \ar[r] & R_{n+1}^{\spin}(X) \ar[r]  & \pos_n(X) \ar[r]   &\Omega_n^\spin (X) \ar[r]  &  }
\] 
where the maps are defined the obvious way.

Now let $X$ be a proper metric space equipped with a proper and cocompact isometric action of a discrete group $\Gamma$. 
\begin{definition}
We denote by $\Omega^{\spin}_n(X)^\Gamma$ the set of bordism classes of pairs $(M, f)$, where $M$ is an $n$-dimensional complete spin manifold (without boundary) such that $M$ is equipped  with a proper and cocompact isometric action of $\Gamma$ and $f: M\to X$ is a $\Gamma$-equivariant  continuous map. Two such pairs $(M_1, f_1)$ and $(M_2, f_2)$ are bordant if there is a bordism $W$ between $M_1$ and $M_2$ (with compatible spin structure), and a $\Gamma$-equivariant continuous map $F: W\to X$ such that $F|_{M_i} = f_i$. Then $\Omega_n^\spin(X)^\Gamma$ is an abelian group with the addition being disjoint union. We call $\Omega_n^\spin(X)^\Gamma$ the $n$-dimensional $\Gamma$-invariant spin bordism of $X$.  
\end{definition}

Similarly, we can define the groups $\pos_n(X)^\Gamma$ and $R_n^\spin(X)^\Gamma$, which fit into the following exact sequence:
\[ 
\xymatrix{  \ar[r] &\Omega^\spin_{n+1}(X)^\Gamma \ar[r] & R_{n+1}^{\spin}(X)^\Gamma \ar[r]  & \pos_n(X)^\Gamma \ar[r]   &\Omega_n^\spin (X)^\Gamma \ar[r]  &  }
\] 
Now suppose $(M, f)\in \Omega^{\spin}_{i}(X)^\Gamma$ and denote by $D_M$ the associated Dirac operator on $M$.   Recall that there is a local index map (cf. Section $\ref{sec:indexmap}$)
\[  \ind_L : K_i^\Gamma(M) \to K_i(C^\ast_L( M)^\Gamma ). \]
Let  $f_\ast: K_i(C^\ast_L(M)^\Gamma ) \to K_i(C^\ast_L( X)^\Gamma )$
be the map induced by $f: M\to X$. We define the map 
\[ \ind_L: \Omega^{\spin}_i(X)^\Gamma \to K_i(C^\ast_L(X)^\Gamma ), \quad (M, f) \mapsto f_\ast [\ind_L(D_{M})]. \] 
Recall that $K_i(C^\ast_L(Y)^\Gamma)$ is naturally isomorphic to $ K_i^\Gamma(Y)$ for all finite dimensional simplicial complex $Y$ equipped with a proper and cocompact isometric action of a discrete group $\Gamma$ (cf. \cite[Theorem 3.2]{MR1451759}). Now the well-definedness of the map $\ind_L$ follows, for example, from the geometric description of $K$-homology groups \cite{BD82} \cite{MR2732043}.  

Now assume in addition we have a positive scalar curvature metric $g$ on $M$. In other words, we have $(M, f, g) \in \pos_i(X)^\Gamma$. Then we have the map
 \[ \rho: \pos_i(X)^\Gamma \to K_i(C^\ast_{L,0}(X)^\Gamma ), \quad (M, f, g) \mapsto  f_\ast[ \rho (D_{M}, g)]. \] 
The well-definedness of the map $\rho$ follows immediately from Theorem $\ref{thm:main}$ above. 

Moreover, suppose we have $(M, f, h)\in R^{\spin}_i(X)^\Gamma$.  By the discussion in Section $\ref{sec:mbpsc}$ and the relative higher index theorem (cf.\cite{UB95},\cite{MR3122162}), we have the following well-defined homomorphism
 \[ \ind: R^{\spin}_i(X)^\Gamma \to K_i(C^\ast(X)^\Gamma ), \quad (M, f, g) \mapsto f_\ast [ \ind(D_{ M})].\] 
Recall that $\ind(D_{M})$ is the index class of $D_{M_\infty}$ in $K_i(C^\ast(M)^\Gamma)$, where $D_{M_\infty}$ is the Dirac operator on  $M_\infty = M\cup_{\partial M}(\mathbb R_{\geq 0}\times \partial M)$ (cf. Claim $\ref{claim:fin}$).

\begin{theorem}\label{thm:stolz}
 For all $n\in \mathbb N$,  the following diagram commutes
\[ \xymatrix{ \Omega^{\spin}_{n+1}(X)^\Gamma \ar[d]^{\ind_L} \ar[r] & R_{n+1}^{\spin}(X)^\Gamma \ar[r] \ar[d]^{\ind} & \pos_n(X)^\Gamma \ar[r] \ar[d]^{\rho}  &\Omega_{n}^{\spin}(X)^\Gamma  \ar[d]^{\ind_L}  \\
 K_{n+1}(C^\ast_L(X)^\Gamma) \ar[r] & K_{n+1} (C^\ast(X)^\Gamma) \ar[r]^{\partial} & K_{n}(C^\ast_{L, 0}(X)^\Gamma) \ar[r] & K_{n}(C^\ast_L(X)^\Gamma)   }
\]

\end{theorem}
\begin{proof}
The commutativities of the first square and the third square follow immediately from the definition. 

Let $(M, f, h)\in R^{\spin}_i(X)^\Gamma$. Then by Theorem $\ref{thm:main}$, we have 
\[  \partial(\ind(D_{M}) )  = \rho(D_{\partial M}, h). \]
This shows the commutativity of the second square. 

\end{proof}

\begin{remark}As we shall see in Section $\ref{sec:coin}$, we have a natural isomorphism between the following two long exact sequences 
\[
\xymatrix {  K_{n+1}( D^\ast(X)^\Gamma ) \ar[r]  \ar[d]^{\cong} & K_{n+1}(D^\ast( X)^\Gamma/C^\ast( X)^\Gamma) \ar[r] \ar[d]^{\cong} & K_{n}(C^\ast(X)^\Gamma )  \ar[r]\ar[d]^{\cong} & K_{n}( D^\ast(X)^\Gamma )  \ar[d]^{\cong} \\
K_{n}( C^\ast_{L, 0}(X)^\Gamma ) \ar[r] & K_{n}(C^\ast_L(X)^\Gamma) \ar[r] & K_{n}(C^\ast(X)^\Gamma )  \ar[r] & K_{n-1}( C^\ast_{L, 0}(X)^\Gamma ) }
\]
Now if in addition the action of $\Gamma$ on $X$ is free, then we have $ \Omega_n^\spin(X)^\Gamma = \Omega_n^\spin(X/\Gamma), $ $\pos_n(X)^\Gamma = \pos_n(X/\Gamma),$
and $R_n^\spin(X)^\Gamma = R_n^\spin(X/\Gamma)$. 
Therefore we see that in this case Theorem $\ref{thm:stolz}$ above recovers \cite[Theorem 1.31]{Piazza:2012fk} of Piazza and Schick. We emphasize that our proof works equally well for both the even and the odd cases. In particular, Theorem $\ref{thm:stolz}$ holds for all $n$. 
\end{remark}

\begin{remark}
We also have the maximal version of Theorem $\ref{thm:stolz}$, by replacing all the $C^\ast$-algebras by their maximal versions.
\end{remark}

\section{Roe algebras and localization algebras}\label{sec:coin}

In this section, we show the natural isomorphism between the long exact sequence of localization algebras (in Theorem $\ref{thm:stolz}$) and the long exact sequence of Higson and Roe. In particular, the explicit construction naturally identifies our definition of the higher rho invariant (see Section \ref{sec:rho} above) with the higher rho invariant of Higson and Roe \cite[Definition 7.1]{MR2761858}. 

Let $X$ be a \emph{finite dimensional simplicial complex} equipped with a proper and cocompact isometric action of a discrete group $\Gamma$. Consider the following short exact sequences:
\[  0 \to C^\ast(X)^\Gamma  \to D^\ast(X)^\Gamma \to D^\ast( X)^\Gamma/C^\ast( X)^\Gamma \to 0,   \]
and 
\[  0 \to C_{L,0}^\ast( X)^\Gamma  \to C_L^\ast(X)^\Gamma \to C^\ast(X)^\Gamma \to 0.   \]
We claim that there are natural homomorphisms  
\[ \beta_{i} : K_i(D^\ast(X)^\Gamma/C^\ast(X)^\Gamma) \to K_{i-1} (C_L^\ast( X)^\Gamma). \]
\textbf{Odd case.} That is, $i$ is an odd integer. Suppose $u$ is an invertible element in $D^\ast(X)^\Gamma/C^\ast( X)^\Gamma$ with its inverse $v$. Let $U, V\in D^\ast(X)^\Gamma$ be certain (not necessarily invertible) lifts of $u$ and $v$. Without loss of generality\footnote{by choosing approximations of $U$ and $V$, hence approximations of $u$ and $v$, if necessary}, we assume that $U$ and $V$ have finite propagations.

Recall that for each $n\in \mathbb N$, there exists a  $\Gamma$-invariant locally finite open cover $\{Y_{n, j}\}$ of  $\ X$  such that diameter$(Y_{n, j})< 1/n$ for all $j$. Let $\{\phi_{n, j}\}$ be a $\Gamma$-invariant smooth partition of unity subordinate to $\{Y_{n, j}\}$. If we write 
\[ U(t) = \sum_{j} \left( (1 - (t-n) ) \phi_{n, j}^{1/2} U \phi_{n, j}^{1/2} + (t-n) \phi_{n+1, j}^{1/2} U \phi_{n+1, j}^{1/2}\right) \]
for $t\in [n, n+1]$. Notice that $[U, \phi_{n, j}^{1/2}] \in C^\ast(X)^\Gamma$ . It follows immediately that the path $U(t)$, $0 \leq t <\infty$,  defines a class in $K_1(D_L^\ast( X)^\Gamma/C_L^\ast( X)^\Gamma).$ Then the map $\beta_1$ is defined by
\[  \beta_1([u]) = \partial_1 [U(t)] \in K_0(C_L^\ast( X)^\Gamma),   \]
where $\partial_1 : K_1(D_L^\ast( X)^\Gamma/C_L^\ast( X)^\Gamma) \to  K_0(C_L^\ast( X))^\Gamma $ is the canonical boundary map (cf. Section $\ref{sec:kt}$) in the six-term $K$-theory exact sequence induced by  
\[  0\to C_L^\ast( X)^\Gamma \to D_L^\ast( X)^\Gamma \to D_L^\ast( X)^\Gamma/ C_L^\ast( X)^\Gamma\to 0.\] 
It is not difficult to verify that  $\beta_1$ is well-defined.
%
%We also define $V(t)$ exactly the same way. 
%
%Now we define
%\begin{align*}
% P(t) &= W(t) \begin{pmatrix} 1 & 0 \\ 0 & 0\end{pmatrix} W(t)^{-1} \\
%&= \begin{pmatrix} U(t)V(t) + U(t)V(t)(1-U(t)V(t)) & (2 + U(t)V(t))(1-U(t)V(t)) U(t) \\ V(t)(1-U(t)V(t)) & (1-U(t)V(t))^2\end{pmatrix},
%\end{align*}
%where \[ W(t) = \begin{pmatrix} 1 & U(t) \\ 0 & 1\end{pmatrix} \begin{pmatrix} 1 & 0 \\ -V(t) & 1\end{pmatrix} \begin{pmatrix} 1 & U(t)  \\ 0 & 1 \end{pmatrix}\begin{pmatrix} 0 & -1\\ 1 & 0 \end{pmatrix}.\]
%Moreover, notice that the propagations of $U(t)$ and $V(t)$ both go to $0$, as $ t\to \infty$. it follows immediately that 
%\[  P(t) - \begin{pmatrix} 1 & 0 \\ 0 & 0 \end{pmatrix} \in C_L^\ast( X)^\Gamma. \]
%The map $\beta_1$ is defined by
%\[ \beta_1 ([u])  = [P(t)] -  \left[\begin{pmatrix} 1 & 0 \\ 0 & 0 \end{pmatrix}\right].  \]

\textbf{Even case.} The even case is parallel to the odd case above.  Suppose $q$ is an idempotent in $D^\ast( X)^\Gamma/C^\ast( X)^\Gamma$. If $Q$ is a lift of $q$, then we define 
\[ Q(t) = \sum_{i} \left( (1 - (t-n) ) \phi_{n, j}^{1/2} Q \phi_{n, j}^{1/2} + (t-n) \phi_{n+1, j}^{1/2} Q \phi_{n+1, j}^{1/2}\right) \]
for $t\in [n, n+1]$ with the same $\phi_{n, j}$ as above. Now the map $\beta_0$ is defined by
\[ \beta_0([q]) = \partial_0 [Q(t) ] \in K_0(C_L^\ast( X)^\Gamma), \]
where $\partial_0 : K_0(D_L^\ast( X)^\Gamma/C_L^\ast( X)^\Gamma) \to  K_1(C_L^\ast( X))^\Gamma $ is the canonical boundary map in the six-term $K$-theory exact sequence induced by  
\[  0\to C_L^\ast( X)^\Gamma \to D_L^\ast( X)^\Gamma \to D_L^\ast( X)^\Gamma/ C_L^\ast( X)^\Gamma\to 0.\]

Now we turn to the maps 
\[ \alpha_{i} : K_i(D^\ast( X)^\Gamma) \to K_{i-1} (C_{L, 0}^\ast( X)^\Gamma). \] 
Suppose  $u$ (resp. $q$) is an invertible element (resp. idempotent) in $D^\ast( X)^\Gamma$. In this case, we can simply choose $U = u$ (resp. $Q = q$).  Let $U(t)$ (resp. $Q(t)$) be as above. Now we apply the construction of the index map $\partial_1$ (resp. $\partial_0$) in Section $\ref{sec:kt}$ to $U(t)$ (resp. $Q(t)$). One immediately sees that the resulting element is an idempotent (resp. invertible element) in 
\[  (C_{L,0}^\ast(X)^\Gamma)^+.\]
We define $\alpha_1(u)$ (resp. $\alpha_0(q)$) to be this element  in  $K_0(C_{L,0}^\ast( X)^\Gamma)$ (resp. $K_1(C_{L,0}^\ast( X)^\Gamma)$). Again it is not difficult to see that $\alpha_i$ is well-defined.

%then  $e^{2\pi i Q}$ is invertible.
%However, in general $e^{2\pi i Q}$ does not have finite propagation. This can be easily fixed by taking a polynomial approximation of $e^{2\pi i z}$, $z\in \mathbb C$, on some bounded subset of $\mathbb C$. More precisely, let $B_r\subset \mathbb C$ be the ball of radius $r$ centered at the origin such that the interior of  $B_r$ contains the spectrum on $Q$. Let $f(z)$ be a polynomial in $z$ on $B_r$ such that 
%\[  \sup_{z\in B_r} \| f(z) - e^{2\pi i z} \| < \varepsilon.   \]
%Then $f(Q)$ is also invertible, when $\varepsilon$ is sufficiently small. 

\begin{proposition}\label{prop:coin}
The homomorphisms 
\[ \alpha_i : K_i(D^\ast( X)^\Gamma) \to K_{i-1} (C_{L, 0}^\ast( X)^\Gamma) \]
and
\[ \beta_i : K_i(D^\ast( X)^\Gamma/C^\ast( X)^\Gamma) \to K_{i-1} (C_L^\ast( X)^\Gamma) \] 
are isomorphisms. Moreover, $\alpha_i$ and $\beta_i$ are natural in the sense that the following diagram commutes.
\[ 
\xymatrix {  K_{i+1}( D^\ast( X)^\Gamma ) \ar[r]  \ar[d]^{\alpha_{i+1}}& K_{i+1}(D^\ast( X)^\Gamma/C^\ast( X)^\Gamma) \ar[r] \ar[d]^{\beta_{i+1}} & K_{i}(C^\ast( X)^\Gamma )  \ar[r]\ar[d]^{=} & K_{i}( D^\ast( X)^\Gamma )  \ar[d]^{\alpha_{i}} \\
K_{i}( C^\ast_{L, 0}( X)^\Gamma ) \ar[r] & K_{i}(C^\ast_L( X)^\Gamma) \ar[r] & K_{i}(C^\ast( X)^\Gamma )  \ar[r] & K_{i-1}( C^\ast_{L, 0}( X)^\Gamma ) }
\]

\end{proposition}
\begin{proof}
The commutativity of the diagram follows by the definitions of the maps $\alpha_i$ and $\beta_i$. So we are only left to prove that $\alpha_i$ and $\beta_i$ are isomorphisms.

We first prove that $\beta_i$ is an isomorphism. Notice that we have Mayer-Vietoris sequences for $K$-theory of both $D^\ast(X)^\Gamma/C^\ast(X)^\Gamma$ and $C_L^\ast(X)^\Gamma$ (cf. \cite[Chapter 5]{MR1399087} and \cite{MR1451759}). It is trivial to verify that $\beta_i$ is an isomorphism when $X$ is $0$-dimensional  (i.e. consisting of discrete points). By using Mayer-Vietoris sequence and the five lemma, the general case follows from induction on the dimension of $X$.

Now it follows from the five lemma that $\alpha_i$ is an isomorphism. This finishes the proof.
   
\end{proof}

\begin{remark}
In fact both  $K_{i+1}(D^\ast(X)^\Gamma/C^\ast(X)^\Gamma)$ and $K_i (C_L^\ast(X)^\Gamma)$ are  naturally identified with the $K$-homology group $K_i^\Gamma(X)$ (cf. \cite[Chapter 5]{MR1399087} and \cite{MR1451759}). The map $\beta_i$ equates these two natural identifications.
\end{remark}

\begin{remark}
%In the proposition above, we saw that there is a natural isomorphism 
%\[ \alpha_i : K_i(D^\ast( X)^\Gamma) \to K_{i-1} (C_{L, 0}^\ast( X)^\Gamma). \]
By following the construction of the isomorphism $\alpha_i$, one immediately sees that our definition of the higher rho invariant (see Section \ref{sec:rho} above) is equivalent to that of Higson and Roe \cite[Definition 7.1]{MR2761858}.
\end{remark}

\begin{remark}
We also have the maximal version of Proposition $\ref{prop:coin}$, by replacing all the $C^\ast$-algebras by their maximal versions.
\end{remark}


\begin{thebibliography}{10}

\bibitem{APS73b}
M.~F. Atiyah, V.~K. Patodi, and I.~M. Singer.
\newblock Spectral asymmetry and {R}iemannian geometry.
\newblock {\em Bull. London Math. Soc.}, 5:229--234, 1973.

\bibitem{A-P-S75a}
M.~F. Atiyah, V.~K. Patodi, and I.~M. Singer.
\newblock Spectral asymmetry and {R}iemannian geometry. {I}.
\newblock {\em Math. Proc. Cambridge Philos. Soc.}, 77:43--69, 1975.

\bibitem{MAIS63}
M.~F. Atiyah and I.~M. Singer.
\newblock The index of elliptic operators on compact manifolds.
\newblock {\em Bull. Amer. Math. Soc.}, 69:422--433, 1963.

\bibitem{MR715325}
S.~Baaj and P.~Julg.
\newblock Th{\'e}orie bivariante de {K}asparov et op{\'e}rateurs non born{\'e}s
  dans les {$C^{\ast} $}-modules hilbertiens.
\newblock {\em C. R. Acad. Sci. Paris S{\'e}r. I Math.}, 296(21):875--878,
  1983.

\bibitem{BCH94}
P.~Baum, A.~Connes, and N.~Higson.
\newblock Classifying space for proper actions and {$K$}-theory of group
  {$C^\ast$}-algebras.
\newblock In {\em {$C^\ast$}-algebras: 1943--1993 ({S}an {A}ntonio, {TX},
  1993)}, volume 167 of {\em Contemp. Math.}, pages 240--291. Amer. Math. Soc.,
  Providence, RI, 1994.

\bibitem{BD82}
P.~Baum and R.~G. Douglas.
\newblock {$K$} homology and index theory.
\newblock In {\em Operator algebras and applications, {P}art {I} ({K}ingston,
  {O}nt., 1980)}, volume~38 of {\em Proc. Sympos. Pure Math.}, pages 117--173.
  Amer. Math. Soc., Providence, R.I., 1982.

\bibitem{MR2732043}
P.~Baum, N.~Higson, and T.~Schick.
\newblock A geometric description of equivariant {$K$}-homology for proper
  actions.
\newblock In {\em Quanta of maths}, volume~11 of {\em Clay Math. Proc.}, pages
  1--22. Amer. Math. Soc., Providence, RI, 2010.

\bibitem{BB98}
B.~Blackadar.
\newblock {\em {$K$}-theory for operator algebras}, volume~5 of {\em
  Mathematical Sciences Research Institute Publications}.
\newblock Cambridge University Press, Cambridge, second edition, 1998.

\bibitem{MR1339924}
B.~Botvinnik and P.~B. Gilkey.
\newblock The eta invariant and metrics of positive scalar curvature.
\newblock {\em Math. Ann.}, 302(3):507--517, 1995.

\bibitem{UB95}
U.~Bunke.
\newblock A {$K$}-theoretic relative index theorem and {C}allias-type {D}irac
  operators.
\newblock {\em Math. Ann.}, 303(2):241--279, 1995.

\bibitem{AC85}
A.~Connes.
\newblock Noncommutative differential geometry.
\newblock {\em Inst. Hautes \'Etudes Sci. Publ. Math.}, (62):257--360, 1985.

\bibitem{AC94}
A.~Connes.
\newblock {\em Noncommutative geometry}.
\newblock Academic Press Inc., San Diego, CA, 1994.

\bibitem{CM90}
A.~Connes and H.~Moscovici.
\newblock Cyclic cohomology, the {N}ovikov conjecture and hyperbolic groups.
\newblock {\em Topology}, 29(3):345--388, 1990.

\bibitem{MR775126}
A.~Connes and G.~Skandalis.
\newblock The longitudinal index theorem for foliations.
\newblock {\em Publ. Res. Inst. Math. Sci.}, 20(6):1139--1183, 1984.

\bibitem{MR2220522}
N.~Higson and J.~Roe.
\newblock Mapping surgery to analysis. {I}. {A}nalytic signatures.
\newblock {\em $K$-Theory}, 33(4):277--299, 2005.

\bibitem{MR2220523}
N.~Higson and J.~Roe.
\newblock Mapping surgery to analysis. {II}. {G}eometric signatures.
\newblock {\em $K$-Theory}, 33(4):301--324, 2005.

\bibitem{MR2220524}
N.~Higson and J.~Roe.
\newblock Mapping surgery to analysis. {III}. {E}xact sequences.
\newblock {\em $K$-Theory}, 33(4):325--346, 2005.

\bibitem{MR2761858}
N.~Higson and J.~Roe.
\newblock {$K$}-homology, assembly and rigidity theorems for relative eta
  invariants.
\newblock {\em Pure Appl. Math. Q.}, 6(2, Special Issue: In honor of Michael
  Atiyah and Isadore Singer):555--601, 2010.

\bibitem{MR1219916}
N.~Higson, J.~Roe, and G.~L. Yu.
\newblock A coarse {M}ayer-{V}ietoris principle.
\newblock {\em Math. Proc. Cambridge Philos. Soc.}, 114(1):85--97, 1993.

\bibitem{MK78}
M.~Karoubi.
\newblock {\em {$K$}-theory}.
\newblock Springer-Verlag, Berlin, 1978.
\newblock An introduction, Grundlehren der Mathematischen Wissenschaften, Band
  226.

\bibitem{GK80}
G.~Kasparov.
\newblock The operator {$K$}-functor and extensions of {$C^{\ast} $}-algebras.
\newblock {\em Izv. Akad. Nauk SSSR Ser. Mat.}, 44(3):571--636, 719, 1980.

\bibitem{GK88}
G.~Kasparov.
\newblock Equivariant {$KK$}-theory and the {N}ovikov conjecture.
\newblock {\em Invent. Math.}, 91(1):147--201, 1988.

\bibitem{BLMM89}
H.~B. Lawson, Jr. and M.-L. Michelsohn.
\newblock {\em Spin geometry}, volume~38 of {\em Princeton Mathematical
  Series}.
\newblock Princeton University Press, Princeton, NJ, 1989.

\bibitem{ELPP01}
E.~Leichtnam and P.~Piazza.
\newblock On higher eta-invariants and metrics of positive scalar curvature.
\newblock {\em $K$-Theory}, 24(4):341--359, 2001.

\bibitem{JLott92}
J.~Lott.
\newblock Higher eta-invariants.
\newblock {\em $K$-Theory}, 6(3):191--233, 1992.

\bibitem{MR0362407}
A.~S. Mi{\v{s}}{\v{c}}enko.
\newblock Infinite-dimensional representations of discrete groups, and higher
  signatures.
\newblock {\em Izv. Akad. Nauk SSSR Ser. Mat.}, 38:81--106, 1974.

\bibitem{Piazza:2012fk}
P.~Piazza and T.~Schick.
\newblock Rho-classes, index theory and {S}tolz's positive scalar curvature
  sequence.
\newblock preprint, 2012.

\bibitem{MR2661442}
Y.~Qiao and J.~Roe.
\newblock On the localization algebra of {G}uoliang {Y}u.
\newblock {\em Forum Math.}, 22(4):657--665, 2010.

\bibitem{MR1147350}
J.~Roe.
\newblock Coarse cohomology and index theory on complete {R}iemannian
  manifolds.
\newblock {\em Mem. Amer. Math. Soc.}, 104(497):x+90, 1993.

\bibitem{MR1399087}
J.~Roe.
\newblock {\em Index theory, coarse geometry, and topology of manifolds},
  volume~90 of {\em CBMS Regional Conference Series in Mathematics}.
\newblock Published for the Conference Board of the Mathematical Sciences,
  Washington, DC, 1996.

\bibitem{Roe:2012kq}
J.~Roe.
\newblock Positive curvature, partial vanishing theorems, and coarse indices.
\newblock arXiv:1210.6100, 10 2012.

\bibitem{JR83}
J.~Rosenberg.
\newblock {$C^{\ast} $}-algebras, positive scalar curvature, and the {N}ovikov
  conjecture.
\newblock {\em Inst. Hautes {\'E}tudes Sci. Publ. Math.}, (58):197--212 (1984),
  1983.

\bibitem{SS95}
S.~Stolz.
\newblock Positive scalar curvature metrics---existence and classification
  questions.
\newblock In {\em Proceedings of the {I}nternational {C}ongress of
  {M}athematicians, {V}ol.\ 1, 2 ({Z}{\"u}rich, 1994)}, pages 625--636, Basel,
  1995. Birkh{\"a}user.

\bibitem{MR1707352}
S.~Weinberger.
\newblock Higher {$\rho$}-invariants.
\newblock In {\em Tel {A}viv {T}opology {C}onference: {R}othenberg
  {F}estschrift (1998)}, volume 231 of {\em Contemp. Math.}, pages 315--320.
  Amer. Math. Soc., Providence, RI, 1999.

\bibitem{MR3122162}
Z.~Xie and G.~Yu.
\newblock A relative higher index theorem, diffeomorphisms and positive scalar
  curvature.
\newblock {\em Adv. Math.}, 250:35--73, 2014.

\bibitem{MR1451759}
G.~Yu.
\newblock Localization algebras and the coarse {B}aum-{C}onnes conjecture.
\newblock {\em $K$-Theory}, 11(4):307--318, 1997.

\bibitem{GY98}
G.~Yu.
\newblock The {N}ovikov conjecture for groups with finite asymptotic dimension.
\newblock {\em Ann. of Math. (2)}, 147(2):325--355, 1998.

\bibitem{MR2732068}
G.~Yu.
\newblock A characterization of the image of the {B}aum-{C}onnes map.
\newblock In {\em Quanta of Maths}, volume~11 of {\em Clay Math. Proc.}, pages
  649--657. Amer. Math. Soc., Providence, RI, 2010.

\end{thebibliography}
\end{document}